\documentclass[leqno,12pt]{article} 

\usepackage{amsmath, amssymb}
\usepackage{amsthm} 
\usepackage{amsfonts,float,tikz,hyperref,layout,mathtools} 
\usetikzlibrary{calc}
\theoremstyle{plain} 
\newtheorem{thm}{\indent\sc Theorem}[section] 
\newtheorem{lem}[thm]{\indent\sc Lemma}

\newtheorem{prop}[thm]{\indent\sc Proposition}
\newtheorem*{cla}{\indent\sc Claim}
\theoremstyle{definition} 
\newtheorem{defi}[thm]{\indent\sc Definition}
\newtheorem*{rem}{\indent\sc Remark}

\makeatletter
\def\address#1#2{\begingroup
	\noindent\parbox[t]{7.8cm}{%
		\small{\scshape\ignorespaces#1}\par\vskip1ex
		\noindent\small{\itshape E-mail address}%
		\/: #2\par\vskip4ex}\hfill%
	\endgroup}%
\makeatother
\title{\uppercase{Some relation between spectral dimension and Ahlfors regular conformal dimension on infinite graphs}} 
\author{
	\textsc{K\^ohei Sasaya} 
}
\date{} 
\newcommand{\qs}{\sim_{\mathrm{QS}}}

\newcommand{\ard}{\dim_\mathrm{ARC}}
\newcommand{\diam}{\mathrm{diam}}

\newcommand{\ol}[1]{\overline{#1}}
\newcommand{\ul}[1]{\underline{#1}}
\newcommand{\Mc}[1]{\mathcal{#1}}
\newcommand{\Mb}[1]{\mathbb{#1}}

\newcommand{\Mr}[1]{\mathrm{#1}}
\renewcommand{\Cup}{\bigcup}
\renewcommand{\Cap}{\bigcap}
\newcommand{\scup}{\sqcup}
\newcommand{\sCup}{\bigsqcup}
\newcommand{\tr}{\triangle}
\newcommand{\mim}{\mathrm{Im}}
\newcommand{\mre}{\mathrm{Re}}

\newcommand{\Etkw}[1]{\mathcal{E}_{2,k,#1}}

\newcommand{\iu}{\sqrt{-1}}

\newcommand{\doi}[1]{\href{https://doi.org/#1}{DOI:\ #1}}

\begin{document}

		\maketitle
\footnote{ 
	2020 \textit{Mathematics Subject Classification}.
	Primary 60J10; Secondary 30L10.
}
\footnote{ 
	\textit{Key words and phrases}.
	Spectral dimension, Ahlfors regular conformal dimension, random walk, resistance metric, quasisymmetry.
}
\footnote{ 
	This work was supported by JSPS KAKENHI Grant Number JP20J23120.
}
	\begin{abstract}
	The spectral dimension $d_s$ of a weighted graph is an exponent associated with the asymptotic behavior of the random walk on the graph. The Ahlfors regular conformal dimension $\ard$ of the graph distance is a quasisymmetric invariant, where quasisymmetry is a well-studied property of homeomorphisms between metric spaces. In this paper, we give a typical example of a fractal-like graph with $d_s<\ard<2$ and prove a sufficient condition for $\ard\le d_s<2.$
	\end{abstract}

\section{Introduction}
This paper aims to evaluate a dimension of an infinite graph as a metric space, defined through quasisymmetric transformations. Let $(X,d)$ be a metric space and $\mu$ be a Borel measure on it. We first recall the definition of quasisymmetry.
\begin{defi}[Quasisymmetry, \cite{Kig2020}]\label{defQS}
	Let $X$ be a set and $d,\delta$ be metrics on $X.$ We say $d$ is \emph{quasisymmetric} to $\delta,$ and write $d\qs\delta,$ if there exists a homeomorphism $\theta:[0,\infty)\to[0,\infty)$ such that for any $x,y,z\in X$ with $x\ne z,$
	\[\delta(x,y)/\delta(x,z)\le \theta \bigl(d(x,y)/d(x,z)\bigr).\]
\end{defi}
For example, $d\qs d^\alpha$ for any $\alpha\in(0,1)$ and a metric space $(X,d).$ The idea of this definition is that any annulus in $(X, d)$ is comparable to one in $(X, \delta)$. 
Beurling and Ahlfors~\cite{BA} implicitly introduced quasisymmetry as a property of a map from $\Mb{R}$ to $\Mb{R},$ and characterized it as the boundary value of a quasiconformal map from the upper half-plane to itself. Kelingos~\cite{Kel} named it quasisymmetry. Tukia and V\"ais\"ala~\cite{TV1980} generalized this notion for embedding maps from one metric space to another.
Quasisymmetry has been studied in various fields, such as Gromov hyperbolic spaces (\cite{BK,BP,MT,Pau} for example) and metric measure spaces (\cite{Hei2001, See2001} for example). Quasisymmetry is also used in studies of heat kernel estimates (see \cite{BCM,BM,KM2020, Kig2012,Mur} for example).\par
Ahlfors regularity and the Ahlfors regular conformal dimension are defined as follows.
\begin{defi}[Ahlfors regularity]\label{defAR}
	For $\alpha>0,$ we say $(X,d)$ is \emph{$\alpha$-Ahlfors regular} if there exists $C>0$ and a Borel measure $\mu$ such that 
	$C^{-1}r^\alpha\le \mu(B_d(x,r)) \le Cr^\alpha$ for any $x\in X$ and $r$ with $\inf_{y\in X\setminus\{x\}}d(x,y)\le r\le\diam(X,d)$
	where $B_d(x,r)=\{y\in X\mid d(x,y)<r\}$ and $\diam(X,d)=\sup_{x,y\in X}d(x,y).$ 
\end{defi}
\begin{defi}[Ahlfors regular conformal dimension]\label{defARC}
The \emph{Ahlfors regular conformal dimension} $\ard$ of $(X,d)$ is defined by 
 \[
	 \ard(X,d)=\inf\biggl\{\alpha\in(0,\infty) \biggm|  
	 \begin{minipage}{180pt}
	 	there exists an $\alpha$-Ahlfors regular metric $\delta$ on $X$ with $d\qs\delta$
	 \end{minipage} \biggr\}.\]
	where $\inf\emptyset=\infty.$
\end{defi}
Note that if $\inf_{x,y}d(x,y)=0$ or $\diam(X,d)=\infty,$ and $(X,d)$ is $\alpha$-Ahlfors regular then $(X,d)$ is not $\beta$-Ahlfors regular for $\beta\ne\alpha.$ In particular, $\alpha$ coincides with the Hausdorff dimension of $(X,d)$ if $(X,d)$ has no isolated points. Usually, we only consider the Ahlfors regularity and $\ard$ in the latter case. However, in the present paper, we consider the case of $\inf_{y}d(x,y)\ne0$ for some $x,$ to treat graph distances.\par
$\ard$ was implicitly introduced for continuous metric spaces by Bourdon and Pajot \cite{BP} and named by Bonk and Kleiner~\cite{BK}. In \cite{BK}, this exponent is related to Cannon's conjecture, which claims that for any hyperbolic group $G$ whose boundary is homeomorphic to the 2-dimensional sphere, there exists a discrete, cocompact, and isometric action of $G$ on the hyperbolic 3-space $\Mb{H}^3.$ Carrasco Piaggio~\cite{Car} characterized $\ard(X,d)$ as a critical value related to the \emph{combinatorial $p$-modulus} $\mathrm{Mod}_p$ of a family of curves $\Gamma$ in a graph $(V, E)$ (approximating $(X,d)$), where
\[\mathrm{Mod}_p(\Gamma)\coloneqq\inf\Bigl\{ \sum_{v\in V}f(v)^p \Bigm| f:V\to[0,\infty),\sum_{v\in\gamma}f(v)\ge1\text{ for any }\gamma\in\Gamma \Bigr\}.\] In recent studies, this characterization of $\ard$ has been also used to construct $p$-Sobolev spaces on fractals (see \cite{Kig2023, Shi}, cf. \cite{MS}).\par
In \cite{Kig2020}, Kigami introduced the notion of a partition satisfying the basic framework (BF-partition for short) and used it to evaluate the Ahlfors regular conformal dimension of compact metric spaces. The author considered this notion on infinite graphs and extended Kigami's results in \cite{Sas21}. Roughly speaking, a BF-partition of an infinite graph is a successive unification of vertices of the given graph with some good conditions. The simplest example is that $(V,E)=(\Mb{Z}_{\ge0},\{(n,m)\mid |n-m|=1\})$ and a BF-partition is defined as (a technical extension of) a map $K:\{(n,k)\mid n,k\ge0\}\to 2^\Mb{C}$ by $K(n,k)=\{m\in\Mb{Z}\mid 2^nk\le m\le 2^n{k+1}\}.$ For another non-trivial example, see the first part of Section \ref{secbf}. For a given partition $K: T\to V,$ we consider the corresponding graph structures in the same way as \cite{Kig2020} and \cite{Sas21}. Roughly speaking, we hierarchically divide $T$ into $\{T_n\}_{n} $ and for each $n$ consider graph structures on $T_n$ such that there is an edge between $w,v\in T_n$ if $K(w)\cap K(v)\ne\emptyset$ and $w\ne v.$ Then some potential theoretic exponents $\ol{d}^s_p(K),$ $\ul{d}^s_p(K)$ of this family of graphs, called the upper and lower $p$-spectral dimensions, for $p>0$ are defined. See Definitions \ref{bf} and \ref{Ce} for the precise definitions of a BF-partition and the $p$-spectral dimensions. For these exponents, the following theorem that is essentially induced by the results in \cite{Kig2020} holds.
\begin{thm}[{\cite[Theorem 4.14]{Sas21}}, cf. {\cite[Theorem 4.7.9]{Kig2020}}]\label{Smain-}Let $d$ be a metric on $V,$ satisfying some properties with respect to $K.$
	\begin{enumerate}
		\item If $p>\ard(V,d)$ then $p>\ol{d}^s_p(K)\ge\ul{d}^s_p(K)\ge \ard(V,d).$
		\item If $p\le\ard(V,d)$ then $p\le\ul{d}^s_p(K)\le\ol{d}^s_p(K)\le\ard(V,d).$
	\end{enumerate}
\end{thm}
See Theorem \ref{Smain} for the precise statement. We emphasize that the $p$-spectral dimensions are combinatorial exponents defined with a partition, and do not have any stochastic characterization. However, as pointed out in \cite{Kig2020} for the case of compact counterparts, if $(V,E)$ is the graphical Sierpi\'nski gasket (see Figure \ref{GSG}) or a graphical generalized Sierpi\'nski carpet (see Figure \ref{GGSC}), $K$ is a canonical partition and $d$ is the graph distance on $(V,E),$ then $\ol{d}^s_2(K)$ and $\ul{d}^s_2(K)$ coincide with the spectral dimension $d_s(V,E)$, defined as follows.
\begin{defi}[Spectral dimension]
	Let $(V,E)$ be a locally finite connected graph and $\mu: E\to(0,\infty)$ be a symmetric weight function (i.e. $\mu(x,y)=\mu(y,x)$ for any $(x,y)\in E$). We also inductively let $\mu(x)=\sum_{y:(x,y)\in E}\mu(x,y),$ $p(0,x,y)=\chi_{\{x\}}(y)/\mu(y)$ and $p(n,x,y)=(\sum_{z:(x,z)\in E}\mu{(x,z)}p(n-1,z,y))/\mu(x).$ The \emph{spectral dimension} of the weighted graph $(V,E,\mu)$ is defined by
	\[d_s(V,E,\mu)=-2\lim_{n\to\infty}(\log p(2n,x,x)/\log n),\]
	if the limit exists for some (or equivalently, for any) $x.$ We simply write $d_s(V,E)$ when $\mu$ is the simple weight i.e. $\mu\equiv1$ on $E.$
\end{defi}
\begin{figure}[tb]
	\centering
	\begin{minipage}[b]{0.45\linewidth}
		\centering
		\begin{tikzpicture}[scale=.15,rotate=210]
			\draw[dashed,thick] (0,0) -- (-5,0) ;
			\clip(0,-20)--(33,0)--(0,20)--cycle;
		\foreach \a in {0,1,2}{
			\foreach \b in {0,1,2}{
				\foreach \c in {0,1,2}{ 
					\foreach \d in {0,1,2}{
						\foreach \e in {0,1,2}{
							\foreach \f in {0,1,2}{
							\coordinate (p\a\b\c\d\e\f) at ({16*cos(120*\a)+8*cos(120*\b)+4*cos(120*\c)+2*cos(120*\d)+cos(120*\e)+cos(120*\f)},{16*sin(120*\a)+8*sin(120*\b)+4*sin(120*\c)+2*sin(120*\d)+sin(120*\e)+sin(120*\f)});
							\filldraw (p\a\b\c\d\e\f) circle [x radius=.2, y radius=.2];
}\draw (p\a\b\c\d\e0)--(p\a\b\c\d\e1)--(p\a\b\c\d\e2)--cycle;
}}}}}
\end{tikzpicture}
\vskip-16pt
		\caption{Part of the graphical Sierpi\'nski gasket}
		\label{GSG}
	\end{minipage}
\quad
	\begin{minipage}[b]{0.45\linewidth}
		\centering\vskip16pt
		\begin{tikzpicture}[scale=0.2]
				\draw[dashed,thick] (32,4.5) -- (36,4.5) ;
					\draw[dashed,thick] (13.5,15) -- (13.5,20) ;
			\clip(-.5,-.5)--(31.5,-.5)--(31.5,14.5)--(-.5,14.5)--cycle;
			\draw(0,0)grid (32,15);
	
\foreach \a in {0,1,2,3}{
	\foreach \b in {0,1}{
		\draw[fill=white] (9*\a+3,9*\b+3)--(9*\a+6,9*\b+3)--(9*\a+6,9*\b+6)--(9*\a+3,9*\b+6)--cycle;
		
}}
\draw[fill=white] (9,9)--(18,9)--(18,15)--(9,15)--cycle;
	\foreach \a in {0,...,32}{
	\foreach \b in {0,...,14}{
		\filldraw (\a,\b) circle  [x radius=0.15, y radius=0.15];
		
}}
\foreach \a in {0,1,2,3}{
	\foreach \b in {0,1}{
		\draw[white,fill=white] (9*\a+3.5,9*\b+3.5)--(9*\a+5.5,9*\b+3.5)--(9*\a+5.5,9*\b+5.5)--(9*\a+3.5,9*\b+5.5)--cycle;
		
}}
\draw[white,fill=white] (9.5,9.5)--(17.5,9.5)--(17.5,17.5)--(9.5,17.5)--cycle;
		\end{tikzpicture}\vskip-8pt
		\caption{Part of the (standard) graphical Sierpi\'nski carpet}
		\label{GGSC}
	\end{minipage}

\end{figure} 
We remark that the value of the limit is independent of the choice of $x$ because $(V,E)$ is connected. Since $p(n,x,y)\mu(y)$ coincides with the transition probability of the weighted random walk on $(V,E,\mu)$ starting from $x,$ the spectral dimension indicates the asymptotic behavior of the return probability of the random walk.\par
In this paper, we consider when same as the cases of the graphical Sierpi\'nski gasket or graphical generalized Sierpi\'nski carpets, the inequalities between the ``geometric dimension'' and the ``stochastic exponent'' $\ard(V,d)\le d_s(V,E,\mu)<2$ or $\ard(V,d)\ge d_s(V,E,\mu)\ge2$
hold for a weighted graph $(V,E,\mu)$ and the graph distance $d.$ (Recall that in the given examples, one of the inequality holds by $\ol{d}^s_2=d_s(V,E)$ and Theorem \ref{Smain-}.) We first construct a graph $(V,E)$ embedded in $\Mb{C}.$ This graph has spatial homogeneity in the following sense: let 
\begin{align}
	S_{n,a,b}=\{z\in \Mb{C}\mid 3^na\le\mre(z)\le3^n(a+1),\ 3^nb\le\mim(z)\le3^n(b+1)\}\label{inab},
\end{align}
then if $\Mr{int}(S_{n,a_j,b_j})\cap V\ne\emptyset$ for some $n,a_j,b_j\ (j=1,2),$ then the restriction of $(V,E)$ to $S_{n,a_1,b_1}$ and to $S_{n,a_2,b_2}$ are isometric (see Figure \ref{Gfstar}). \begin{figure}[tb]
	\centering
	\begin{tikzpicture}[scale=0.12]
		\foreach \d / \D in {0/0, 54/0, 27/27}{		
			\foreach \c / \C in {0/0, 18/0, 9/9, 0/18, 18/18}{	
				\foreach \b /\B in {0/0, 6/0, 3/3, 0/6, 6/6}{
					\foreach \a / \A in {0/0, 1/0, 2/0, 2/1, 2/2, 1/2, 0/2, 0/1}{
						\draw (\a+\b+\c+\d,\A+\B+\C+\D) -- (\a+\b+\c+\d+0.5,\A+\B+\C+\D+0.5) -- (\a+\b+\c+\d+1, \A+\B+\C+\D+1) ;
						\draw (\a+\b+\c+\d,\A+\B+\C+\D+1) -- (\a+\b+\c+\d+0.5,\A+\B+\C+\D+0.5) -- (\a+\b+\c+\d+1, \A+\B+\C+\D) ;
		}}}}
		\draw[dotted,thick] (55,55) -- (60,60) ;
		\draw[dotted,thick] (26,55) -- (21,60) ;
		\draw[dashed, ->, >=stealth] (-5, 0) -- (85,0) node[right] {$\mre$} ;
		\draw[dashed, ->, >=stealth] (0, -5) -- (0,60) node[above] {$\mim$} ;
		\coordinate[label=below left:0] (0) at (0,0);
	\end{tikzpicture}
	\caption{Part of $(V, E)$ around the origin}
	\label{Gfstar}
\end{figure} The first main result of this paper is to show that neither of the considered inequalities holds for this $(V,E)$ with the graph distance $d$ and the simple weight $\mu.$ See Theorem \ref{theex} for the precise statement.\par
The proof of Theorem \ref{theex} implies that some type of symmetry for scaling is sufficient to show $\ol{d}^s_2=d_s(V,E)$ and one of the considered inequalities. To justify this idea, we introduce the resistance on a weighted graph. 
\begin{defi}[Resistance]
	Let $(V',E',\mu)$ be a weighted graph, then for $A,B\subset V'$ with $A\cap B\ne\emptyset$ the \emph{resistance} between $A$ and $B$ is defined by
\[R_\mu(A,B)=\bigl(\inf\bigl\{\frac{1}{2}\sum_{(x,y)\in E'}(f(x)-f(y))^2\mu(x,y)\bigm\vert f:G\to\Mb{R}, f|_A\equiv1, f|_B\equiv 0\bigr\}\bigr)^{-1}\]
\end{defi}
It is known that $(R_\mu(A,B))^{-1}$ attains the minimum and $R_\mu(\{x\},\{y\})$ is the distance on $V'$ (see \cite{Kig2001} for example). Our second main theorem is the following.
\begin{thm}\label{theth}
	Let $(V',E')$ be an infinite, connected, locally finite graph, $d'$ be the graph distance of $(V',E')$ and $\mu$ be a weight on $E$ such that
	\begin{align}\tag{$p_0$}\label{p0}
		\text{for some }p_0>0,\ \mu(x,y)/\mu(x,z)\ge p_0 \text{ for any }x,y,z\in V'\text{ with }(x,y),(x,z)\in E.
	\end{align} If there exist $\alpha,\beta, C>0$ such that $\alpha+\beta>2,$
$C^{-1}d'(x,y)^\alpha\le R_\mu(\{x\},\{y\})\le Cd'(x,y)^\alpha$ and
		$C^{-1}n^\beta\le \mu(B_d(x,n))\le Cn^\beta$
	for any $x,y\in V'$ and $n\ge 0,$ then the limit $d_s(V',E',\mu)$ exists and
	\[\ard(V',d')\le d_s(V',E',\mu)<2.\]
\end{thm}
Here $d_s(V',E',\mu)<2$ follows from the assumption for $R_\mu(\{x\},\{y\}),$ so we emphasize that in this theorem we only treat the case that the associated random walk is recurrent. \par
We remark that in the forthcoming paper \cite{Sas23+} the author defines a variation $\ol{d_s}$ of $d_s$ and proves $\ard(X,\delta)\le \ol{d_s}<2$ when $(X,\delta)$ is a (continuous) low dimensional fractal, even if it is not symmetric for scaling. In \cite{Sas23+}, the counterpart of Theorem \ref{theex} is also proved using resistance estimates in the present paper.\par
The structure of this paper is as follows. We define our targeting graph $(V,E),$ state the first main result, and evaluate resistances on $(V,E)$ in Section \ref{secRes}. Section \ref{secbf} is devoted to introducing the notion of a BF-partition and related results, and then we evaluate $\ard(V,d)$ using these results in Section \ref{secARC}. Finally, we prove Theorem \ref{theth} in Section \ref{secSym}.
\subsection*{Acknowledgments}
I would like to thank my supervisor of the doctoral thesis, Professor Takashi Kumagai for helpful advice about the structure of this article.\par
This work was supported by Japan Society for the Promotion of Science (JSPS) KAKENHI Grant Number JP20J23120. 
\subsection*{Notation}

\begin{itemize}
	\item For a set $X,$ $\# X$ denotes the cardinality of $X.$
	\item $a\vee b$ (resp. $a\wedge b$) denotes $\max\{a,b\}$ (resp. $\min\{a,b\}$).
	\item Let $f,g$ be functions on a set $X$ and $A\subset X.$ We say $f\lesssim g$ (resp. $f\gtrsim g$) for any $x\in A$ if there exists $C>0$ such that $f(x)\le Cg(x)$ (resp. $f(x)\ge Cg(x)$) for any $x\in A.$ We also write $f\asymp g$ (for any $x\in A$) if $f\lesssim g$ and $f\gtrsim g.$
	\item Let $X$ be a set and $f:X\to X$ be a map, then we write $f^k$ instead of $\overbrace{f\circ\cdots\circ f}^{k}.$ Moreover, $f^{-k}$ denotes $(f^{-1})^k$ for $k>0.$
	\item For $A\subset \Mb{C}$ and $\alpha,\beta\in\Mb{C},$ $\alpha A+\beta$ denotes the set $\{\alpha z+\beta\mid z\in A\}.$
	\item $\scup_{\lambda\in\Lambda}A_\lambda$ denotes the disjoint union, that is,
	$\cup_{\lambda\in\Lambda}A_\lambda$ with $A_\lambda\cap A_\tau=\emptyset$ for any $\lambda, \tau\in\Lambda$ with $\lambda\ne\tau.$
	\item Let $\Theta$ be a variable defined by the minimum or maximum of some functions. We say $f$ is the optimal function for $\Theta$ if $f$ attains the minimum or maximum. For example, we say $f$ is the optimal function for $R_\mu(A,B)$ if $(R_\mu(A,B))^{-1}=(1/2)\sum_{(x,y)\in E}(f(x)-f(y))^2\mu(x,y),$ $f|_A\equiv 1$ and $f|_B\equiv 0.$
	\item Let $(V,E,\mu)$ be a weighted graph. We treat $\mu$ as a discrete measure on $V$ i.e. $\mu(A)\coloneqq\sum_{x\in A}\mu(x)=\sum_{x\in A}\sum_{y:(x,y)\in E}\mu(x,y)$ for $A\subset V.$
	\item $\mu_{(V,E)}$ denotes the simple weight on $(V,E).$ For simplicity, we write $R_{(V,E)}$ instead of $R_{\mu_{(V,E)}},$ same as the case of $d_s$.
	\item We abuse the notation $x$ instead of $\{x\}$ if no confusion may occur. For instance, we write $R(x,y)$ instead of $R(\{x\},\{y\}).$
\end{itemize}

\section{Resistance estimate}\label{secRes}

In this section, we construct a fractal-like infinite graph $(V,E)$ appearing in the first main result, evaluate resistances on $(V,E)$ using combinatorial arguments, and calculate $d_s(V,E)$ with the evaluation.\par
Let $S=\{z\mid |\mre(z)|\vee|\mim(z)|\le1/2\}\subset \Mb{C},$
\begin{align*}
	p_j=&\begin{cases}
		0&\text{if }j=0,\\
		\exp(j\pi\iu/4)/\sqrt{2}&\text{if }j=1,3,5,7,\\
		\exp(j\pi\iu/4)/2&\text{if }j=2,4,6,8,
	\end{cases}
\end{align*}
$\varphi_j(z)=(z-p_j)/3+p_j,\ \Phi_0(A)=\Cup_{j=0,1,3,5,7}\varphi(A),\ \Phi_1(A)=\Cup_{j=1}^8\varphi(A)$ for $A\subset\Mb{C}$ (see Figures \ref{pts}, \ref{Gn0}, \ref{Gn1}).
\begin{figure}[tb]
	\centering
	\begin{minipage}{0.32\linewidth}
		\centering
		\begin{tikzpicture}[scale=0.3]
			\draw[ ->, >=stealth] (-5, 0) -- (5,0) node[right] {$\mre$} ;
			\draw[, ->, >=stealth] (0, -5) -- (0,5) node[above] {$\mim$} ;
			\draw[dotted] (-4,-4) -- (4,-4) -- (4,4) -- (-4,4) -- cycle ;
			\coordinate[label=below left:$p_0$] (p0) at (0,0);
			\coordinate[label=below left:$p_1$] (p1) at (4,4);
			\coordinate[label=below left:$p_2$] (p2) at (0,4);
			\coordinate[label=below left:$p_3$] (p3) at (-4,4);
			\coordinate[label=below left:$p_4$] (p4) at (-4,0);
			\coordinate[label=below left:$p_5$] (p5) at (-4,-4);
			\coordinate[label=below left:$p_6$] (p6) at (0,-4);
			\coordinate[label=below left:$p_7$] (p7) at (4,-4);
			\coordinate[label=below left:$p_8$] (p8) at (4,0);
			\foreach \a in {0,1,...,8}
			\filldraw (p\a) circle [x radius=0.2, y radius=0.2];
			\coordinate[label=below: $\frac{1}{2}$] (set) at (4,0);
		\end{tikzpicture}
		\vskip-16pt
		\caption{$\{p_j\}_{j=0}^8$}
		\label{pts}
	\end{minipage}
	\begin{minipage}{0.32\linewidth}
		\centering\vskip16pt
		\begin{tikzpicture}[scale=0.3]
			\draw[dashed, ->, >=stealth] (-5, 0) -- (5,0) ;
			\draw[dashed, ->, >=stealth] (0, -5) -- (0,5) ;
			\draw[dotted] (-4,-4) -- (4,-4) -- (4,4) -- (-4,4) -- cycle ;
			\coordinate(p0) at (0,0);
			\coordinate (p1) at (4,4);
			\coordinate(p2) at (0,4);
			\coordinate (p3) at (-4,4);
			\coordinate (p4) at (-4,0);
			\coordinate (p5) at (-4,-4);
			\coordinate(p6) at (0,-4);
			\coordinate (p7) at (4,-4);
			\coordinate (p8) at (4,0);
			\foreach \a in {0,1,3,5,7}
			{\draw[fill=white]($(p\a)!0.3333!(-4,-4)$) --($(p\a)!0.3333!(4,-4)$) --($(p\a)!0.3333!(4,4)$) --($(p\a)!0.3333!(-4,4)$) --cycle;
			}
		\end{tikzpicture}\vskip-8pt
		\caption{$\Phi_0(S)$}
		\label{Gn0}
	\end{minipage}
	\begin{minipage}{0.32\linewidth}
		\centering\vskip16pt
		\begin{tikzpicture}[scale=0.3]
			
			\draw[dashed, ->, >=stealth] (-5, 0) -- (5,0)  ;
			\draw[dashed, ->, >=stealth] (0, -5) -- (0,5) ;
			\draw[dotted] (-4,-4) -- (4,-4) -- (4,4) -- (-4,4) -- cycle ;
			\coordinate(p0) at (0,0);
			\coordinate (p1) at (4,4);
			\coordinate(p2) at (0,4);
			\coordinate (p3) at (-4,4);
			\coordinate (p4) at (-4,0);
			\coordinate (p5) at (-4,-4);
			\coordinate(p6) at (0,-4);
			\coordinate (p7) at (4,-4);
			\coordinate (p8) at (4,0);
			\foreach \a in {1,...,8}
			{\draw[fill=white]($(p\a)!0.3333!(-4,-4)$) --($(p\a)!0.3333!(4,-4)$) --($(p\a)!0.3333!(4,4)$) --($(p\a)!0.3333!(-4,4)$) --cycle;
				;}
		\end{tikzpicture}\vskip-8pt
		\caption{$\Phi_1(S)$}
		\label{Gn1}
	\end{minipage}
\end{figure} We define $V_n\ (n\ge0)$ by 
\begin{align}
	V_0&=\{p_0, p_1, p_3, p_5, p_7\},\quad V_n=\Phi_{F(n)}\circ\cdots\circ\Phi_{F(1)}V_0\ (n\ge1)\notag,\\	
	\shortintertext{where}
	F(n)=&\begin{cases}
		1 & \text{if }k^2(k-1)<n\le k^3\text{ for some }k\in\Mb{N},\\
		0 & \text{otherwise},
	\end{cases}\notag
\end{align}
and a graph $(V,E)$ by
\begin{equation}\label{defve}
	V=\Cup_{n\ge0}3^n\biggl(V_n+\frac{1+\iu}{2}\biggr),\quad E=\{(x,y)\in V\times V \mid |x-y|=2^{-1/2}\}.
\end{equation}
(see Figure \ref{Gfstar}). For the rest of this paper, $(V,E)$ denotes the graph defined above and $d$ denotes the graph distance of $(V,E)$. Then our first main result is stated as follows.
\begin{thm}\label{theex}
	$$ d_s(V,E)=2\frac{\log5}{\log3+\log5}<\ard(V,d)=\ard(\Mr{SC},|\cdot|_\Mb{C})<2, $$
	where $\Mr{SC}$ is the unique nonempty compact subset of $\Mb{C}$ with $\Phi_1(\Mr{SC})=\Mr{SC},$ called the (standard) Sierpi\'nski carpet (see Figure \ref{FigSC}).
\end{thm}
\begin{rem}
	$2\log5/(\log3+\log5)=\ard(\Mr{Vic},|\cdot|_\Mb{C})$ where $\Mr{Vic}$ is the unique nonempty compact subset of $\Mb{C}$ with $\Phi_0(\Mr{SC})=\Mr{SC},$ called the Vicsek set or the Vicsek tree (see Figure \ref{FigVic}). On the other hand, the value of $\ard(\Mr{SC},|\cdot|_\Mb{C})$ is not known.
\end{rem}
\begin{figure}[tb]
	\centering
	\begin{minipage}[b]{0.45\linewidth}
		\centering
		\begin{tikzpicture}[scale=.015]
			\filldraw (0,0)--(0,243)--(243,243)--(243,0)--cycle;
				\foreach \a in {0,...,80}{
					\foreach \b in {0,...,80}{
				\draw[white,very thin,fill=white] (3*\a+1.2,3*\b+1.2)--(3*\a+1.2,3*\b+1.8)--(3*\a+1.8,3*\b+1.8)--(3*\a+1.8,3*\b+1.2)--cycle;
			}}
		\foreach \a in {0,...,26}{
			\foreach \b in {0,...,26}{
				\draw[fill=white] (9*\a+3,9*\b+3)--(9*\a+3,9*\b+6)--(9*\a+6,9*\b+6)--(9*\a+6,9*\b+3)--cycle;
		}}
		\foreach \a in {0,...,8}{
		\foreach \b in {0,...,8}{
			\draw[fill=white] (27*\a+9,27*\b+9)--(27*\a+9,27*\b+18)--(27*\a+18,27*\b+18)--(27*\a+18,27*\b+9)--cycle;
	}}
	\foreach \a in {0,...,2}{
	\foreach \b in {0,...,2}{
		\draw[fill=white] (81*\a+27,81*\b+27)--(81*\a+27,81*\b+54)--(81*\a+54,81*\b+54)--(81*\a+54,81*\b+27)--cycle;
}}
	\draw[fill=white] (81,81)--(81,162)--(162,162)--(162,81)--cycle;
		\end{tikzpicture}
		\caption{(Standard) Sierpi\'nski carpet }
		\label{FigSC}
	\end{minipage}
	\begin{minipage}[b]{0.45\linewidth}
		\centering
		\begin{tikzpicture}[scale=.015]
			\draw (-121.5,-121.5)--(121.5,121.5);
			\draw (121.5,-121.5)--(-121.5,121.5);
					\foreach \a in {0,1,3,5,7}{
							\foreach \b in {0,1,3,5,7}{
									\foreach \c in {0,1,3,5,7}{ 
											\foreach \d in {0,1,3,5,7}{
													\foreach \e in {1,3,5,7}{
								
																\coordinate (p\a\b\c\d\e) at 
										({(cos(180*\a)-1)*-.5*81*sqrt(2)*cos(45*\a)+(cos(180*\b)-1)*-.5*27*sqrt(2)*cos(45*\b)+(cos(180*\c)-1)*-.5*9*sqrt(2)*cos(45*\c)+(cos(180*\d)-1)*-.5*3*sqrt(2)*cos(45*\d)+.5*3*sqrt(2)*cos(45*\e)},{
										(cos(180*\a)-1)*-.5*81*sqrt(2)*sin(45*\a)+(cos(180*\b)-1)*-.5*27*sqrt(2)*sin(45*\b)+(cos(180*\c)-1)*-.5*9*sqrt(2)*sin(45*\c)+(cos(180*\d)-1)*-.5*3*sqrt(2)*sin(45*\d)+.5*3*sqrt(2)*sin(45*\e)});						
									}\draw (p\a\b\c\d1)--(p\a\b\c\d5);
								\draw (p\a\b\c\d3)--(p\a\b\c\d7);
								}}}}
		\end{tikzpicture}
		\caption{Vicsek set}
		\label{FigVic}
	\end{minipage}
\end{figure} 
Now we proceed to evaluate resistances on $(V,E)$ and other associated graphs. Let
\begin{align*}
	E_n&=\{(x,y)\in V_n\times V_n \mid |x-y|=3^{-n}2^{-1/2}\}\label{defEn}&&\\
	S_T&=\{z\in S \mid \mim(z)=1/2\},&S_B&=\{z\in S \mid \mim(z)=-1/2\},\\
	R_{n,\mathrm{TB}}&=R_{(V_n,E_n)}(V_n\cap S_T, V_n\cap S_B),&
	R_{n,\mathrm{Pt}}&=R_{(V_n,E_n)}(p_1,p_5).
\end{align*}

\begin{thm}\label{res}
	\begin{enumerate}
		\item $R_{n,\mathrm{TB}}\asymp R_{n,\mathrm{Pt}}$ for any $n\ge0.$
		\item Let $m_1(n)=\#\{k\mid k\le n, F(k)=1\}$ and $m_2(n)=\#\{k\mid k\le n, F(k)=1, F(k-1)=0\}$ for $n\ge1.$ 
		Then there exist $C_1,C_2>0$ and $\rho>1$ such that
		\[\rho^{m_1(n)}3^{n-m_1(n)}C_1^{m_2(n)}\le R_{n,\mathrm{Pt}}\le 2\rho^{m_1(n)}3^{n-m_1(n)}C_2^{m_2(n)}\text{ for any }n\ge0. \]
		\item 	There exists $M>0$ such that $R_{n,\mathrm{Pt}}\le R_{n+M,\mathrm{Pt}}/2$ for any $n\ge0.$
	\end{enumerate}
\end{thm} 
\begin{rem}
 Theorem \ref{res}.3 does not follows from Theorem \ref{res}.2. Indeed, we do not know whether $C_1$ equal to $C_2$ for Theorem \ref{res}.2.
\end{rem}
In order to prove the above theorem, we introduce known results about resistance estimate for the graphical Sierpi\'nski carpet and the idea of a flow on a graph. Let $V^{\mathrm{SC}}_0=\{p_1,p_3,p_5,p_7\}$ and $V^{\Mr{SC}}_n=(\Phi_1)^n(V^{\mathrm{SC}}_0).$ Similar to the case of $(V,E),$ we set
\begin{align*}
	E^{\Mr{SC}}_n&=\{(x,y)\in V^{\mathrm{SC}}_n\times V^{\mathrm{SC}}_n \mid |x-y|=3^{-n}\},\\
	R^{\mathrm{SC}}_{n,\Mr{TB}}&=R_{(V^{\mathrm{SC}}_n,E^{\mathrm{SC}}_n)}(V^{\mathrm{SC}}_n\cap S_T, V^{\mathrm{SC}}_n\cap S_B), \\
	R^{\mathrm{SC}}_{n,\Mr{Pt}}&=R_{(V^{\mathrm{SC}}_n,E_n^{\mathrm{SC}})}(p_1,p_5).
\end{align*}
We use the following fact.
\begin{prop}\label{resSC}
	For some $\rho>1,$ $\rho^n\asymp R^{\mathrm{SC}}_{n,\Mr{TB}}\asymp R_{n,\Mr{Pt}}^{\mathrm{SC}}$ for any $n\ge0.$
\end{prop}
\begin{rem}
	\begin{enumerate}
		\item $\rho^n\asymp R_{n,\Mr{TB}}^{\mathrm{SC}}$ follows in the same way as \cite{BB90}. On the other hand, we can not apply this method to show $\rho^n\asymp R_{n,\Mr{Pt}}^{\mathrm{SC}}$ because this method, especially a part of a potential theoretic argument in \cite[Section 4]{BB90}, uses the fact that the optimal function $g$ for $R_{n,\Mr{TB}}^{\mathrm{SC}}$ satisfies $g\equiv 1$ on $V_n^{\mathrm{SC}}\cap S_T$: we can not obtain $R_{n+k,\Mr{Pt}}^{\mathrm{SC}}\gtrsim R_{n,\Mr{Pt}}^{\mathrm{SC}} R_{k,\Mr{Pt}}^{\mathrm{SC}}$ but only $R_{n+k,\Mr{Pt}}^{\mathrm{SC}}\gtrsim R_{n,\Mr{Pt}}^{\mathrm{SC}} R_{k,\Mr{TB}}^{\mathrm{SC}}$ for any $n,k\ge0$ in this way. We will prove $R_{n,\Mr{TB}}^{\mathrm{SC}}\gtrsim R_{n,\Mr{Pt}}^{\mathrm{SC}}$ in Appendix \ref{app}.
		\item The heat kernel estimates of the simple random walks on graphical Sierpi\'nski carpets have been studied in \cite{BB99}, including transient cases without resistance estimates. We can also obtain $\rho^n\asymp R^{\mathrm{SC}}_{n,\Mr{Pt}}$ by \cite[Theorem 1.5]{BB99} and \cite[Theorem 1.3]{BCK2005}.
	\end{enumerate}	
\end{rem}

\begin{defi}[Unit flow]
	Let $(V',E')$ be a connected, locally finite graph. For $A,B\subset V'$ with $A\cap B=\emptyset,$ $f:E'\to \Mb{R}$ is called a unit flow from $A$ to $B$ if $f$ satisfies
	\begin{itemize}
		\item $f(x,y)=-f(y,x)$ for any $(x,y)\in E',$
		\item $\sum_{y:(x,y)\in E'}f(x,y)=0$ for any $x\not\in A\cup B,$
		\item $\sum_{x\in A}\sum_{y:(x,y)\in E'}f(x,y)=1$ and $\sum_{x\in B} \sum_{y:(x,y)\in E'}f(x,y)=-1.$
	\end{itemize}
	Let $\mu$ be a weight on $E',$ then it is known that
	\begin{equation}\label{fl}
		R_\mu(A,B)=\min\biggl\{\frac{1}{2}\sum_{(x,y)\in E'}\frac{f(x,y)^2}{\mu(xy)}\biggm\vert f\text{ is a unit flow from }A\text{ to }B\biggr\}.
	\end{equation}
	
	We say $f$ is an optimal flow for $R_\mu(A,B)$, or an optimal flow from $A$ to $B$ if $f$ is an optimal function for the right hand side of \eqref{fl}.\par
	Before the proof of Theorem \ref{res},  we show some technical lemmas.
\end{defi}
\begin{lem}\label{rpt}
	\begin{enumerate}
		\item $(1/4)R_{n,\Mr{Pt}}\le R_{(G_n,E_n)}(p_1,p_3)\le 4R_{n,\Mr{Pt}}$ for any $n\ge0.$
		\item $R_{n+1,\Mr{Pt}}\ge R_{n,\Mr{Pt}}$ for any $n\ge0.$ In particular, if $F(n+1)=0$ then $R_{n+1,\Mr{Pt}}=3R_{n,\Mr{Pt}}.$
	\end{enumerate}
\end{lem}

\begin{proof}Fix any $n\ge0.$ Let $g_{jk}$ be the optimal functions for $R_{(V_n,E_n)}(p_j,p_k)\ (1\le j,k\le 8).$ 
	\begin{enumerate} 
		\item By symmetry, $g_{13}(z)=1/2$ for any $z\in V_n$ with $\mre(z)=0.$ This and minimality of the potential assure $g_{13}(p_7)\ge1/2$ and $g_{13}(p_5)\le1/2,$ so 
		\[f_1(z):=(g_{13}(z)+g_{75}(z)-g_{13}(p_5))\wedge1\]
		satisfies $f_1(p_1)=f_1(p_7)=1$ and $f_1(p_3)=f_1(p_5)=0.$ Therefore
		\begin{align*}
			(R_{n,\Mr{Pt}})^{-1}\le R_{(V_n,E_n)}(\{p_1,p_7\},\{p_3,p_5\})^{-1}
			\le4R_{(V_n,E_n)}(p_1,p_3)^{-1}.
		\end{align*}
		On the other hand, since $g_{15}(z)=1/2$ for any $z\in V_n$ with $\mim(z)=-\mre(z)$, 
		\[R_{(V_n,E_n)}(p_1,p_3)^{-1}\le R_{(V_n,E_n)}(\{p_1,p_7\},\{p_3,p_5\})^{-1}\le4(R_{n,\Mr{Pt}})^{-1}\]
		similarly follows.
		\item Let $f_2: V_{n+1}\to\Mb{R}$ such that
		\[f_2(z)=\begin{cases}
			g_{15}(\varphi_1^{-1}(z)) & \text{if }\mre(z)+\mim(z)\ge\frac{2}{3}, \\
			g_{15}(\varphi_5^{-1}(z)) & \text{if }\mre(z)+\mim(z)\le-\frac{2}{3}, \\
			\frac{1}{2} & \text{otherwise.}
		\end{cases}\]
		Then $R_{n+1,\Mr{Pt}}\ge 2(\sum_{(x,y)\in E_{n+1}}(f_2(x)-f_2(y))^2)^{-1}=R_{n,\Mr{Pt}}$ follows immediately. If $F(n+1)=0,$ then it is easy to check that $R_{n,\Mr{Pt}}=R_{(V_{n+1},E_{n+1})}(p_1, \varphi_1(p_5))=R_{(V_{n+1},E_{n+1})}(\varphi_1(p_5), \varphi_5(p_1))=R_{(V_{n+1},E_{n+1})}(\varphi_5(p_1), p_5),$ so $R_{n+1,\Mr{Pt}}\le3R_{n,\Mr{Pt}}$ by the triangle inequality. On the other hand, let $f_3: V_{n+1}\to\Mb{R}$ defined by
			\[f_3(z)=\begin{cases}
			\frac{1}{3}g_{15}(\varphi_1^{-1}(z))+\frac{2}{3} & \text{if }z\in \varphi_1(S), \\
			\frac{1}{3}g_{15}(\varphi_0^{-1}(z))+\frac{1}{3} & \text{if }z\in \varphi_0(S), \\
			\frac{1}{3}g_{15}(\varphi_5^{-1}(z)) & \text{if }z\in \varphi_5(S), \\
			\frac{1}{2} & \text{otherwise.}
		\end{cases}\]
		Then $R_{n+1,\Mr{Pt}}\ge 2(\sum_{(x,y)\in E_{n+1}}(f_3(x)-f_3(y))^2)^{-1}=3R_{n,\Mr{Pt}}$ follows. 
	\end{enumerate}
\end{proof}
\begin{lem}\label{cur}
	Let $l=l(n):=\max\{m\mid m\le n, F(m)=0\},$ then
	$R_{n-l,\Mr{Pt}}^{\mathrm{SC}}R_{l,\Mr{Pt}}\gtrsim R_{n,\mathrm{Pt}}$
	for any $n.$
\end{lem}
\begin{proof}
	For any $x,y\in V_n$ with $(x,y)\in E^{\mathrm{SC}}_{n-l},$ there exists a unit flow $f_{xy}$ from $x$ to $y$ on $E_n$ such that $f_{xy}(z,w)=0$ for any $(z,w)\in E_n$ with $|x-z|>2^{1/2}3^{-(n-l)},$ and $(1/2)\sum_{(z,w)\in E_n}f_{xy}(z,w)^2\le 4R_{l,\Mr{Pt}}$ by Lemma \ref{rpt}.1. Let $g$ be the optimal flow for $R^{\mathrm{SC}}_{n-l,\Mr{Pt}}.$ We define $f:E_n\to \Mb{R}$ by 
	\[f(z,w)=\frac{1}{2}\sum_{(x,y)\in E^{\mathrm{SC}}_{n-l}}g(x,y)f_{x,y}(z,w),\]
	then we can see that $f$ is a unit flow from $p_1$ to $p_5.$ We also let
	\[L:=\sup_{n\ge0, z\in V_n}\#\{(x,y)\in E^{\mathrm{SC}}_{n-l}\mid  |x-z|\le2^{1/2}3^{-(n-l)}\},\]
	then
	\begin{align*}
		R_{n,\mathrm{Pt}}&\le\frac{1}{2}\sum_{(z,w)\in E_n}f(z,w)^2
		\le \frac{L}{8}\sum_{(x,y)\in E^{\mathrm{SC}}_{n-l} } g(x,y)^2 \sum_{(z,w)\in E_n}f_{xy}(z,w)^2\\
		&\le \frac{L}{2}R^{\mathrm{SC}}_{n-l,\Mr{Pt}}R_{l,\Mr{Pt}}.
	\end{align*}
	Since $L<\infty,$ the claim follows.
	\end{proof}

	\begin{proof}[Proof of Theorem \ref{res}]
		\begin{enumerate}
			\item It is sufficient to show $R_{n,\mathrm{TB}}\gtrsim R_{n,\mathrm{Pt}}$ for $n\ge1.$
			\begin{itemize}
				\item The case of $F(n)=0$: let $g_{n-1}$ be the optimal function for $R_{n-1,\Mr{Pt}}.$ Here we define $f_n:V_n\to\Mb{R}$ by
				\[f_n(z)=
				\begin{cases}
					1 & \text{if }z\in \varphi_1(S)\cup\varphi_3(S),\\
					g_{n-1}(3z)+g_{n-1}(-3\iu z)-\frac{1}{2} & \text{if }z\in\varphi_0(S),\\
					0 & \text{if }z\in \varphi_5(S)\cup\varphi_7(S).
				\end{cases}\]
				Then $(R_{n,\Mr{TB}})^{-1}\le(1/2)\sum_{(x,y)\in E_n}(f_n(x)-f_n(y))^2\le4(R_{n-1,\Mr{Pt}})^{-1}=12(R_{n,\mathrm{Pt}})^{-1}$ by Lemma \ref{rpt}.2, this completes the case. 
				\item The case of $F(n)=1$: let $l(n)$ be the same as in Lemma \ref{cur}. Using the former case, Proposition \ref{resSC} and Lemma \ref{cur}, we obtain
				$R^{\mathrm{SC}}_{n-l,\Mr{TB}}R_{l,\Mr{TB}}\gtrsim R^{\mathrm{SC}}_{n-l,\Mr{Pt}}R_{l,\Mr{Pt}}\gtrsim R_{n,\mathrm{Pt}}$ for any $n$ with $F(n)=1.$ On the other hand, we can get $R_{n,\mathrm{TB}}\gtrsim R^{\mathrm{SC}}_{n-l,\Mr{TB}}R_{l,\Mr{TB}}$ for any $n$ by the same potential theoretic argument as in \cite[Theorem4.3]{BB90}, which completes the proof.
			\end{itemize}
			\item We inductively prove the upper bound. By Theorem \ref{res}.1, Proposition \ref{resSC} and Lemma \ref{cur}, there exist $C_3,C_4>0$ such that  $C_3R^{\mathrm{SC}}_{n-l,\Mr{Pt}}R_{l,\Mr{Pt}}\ge R_{n,\mathrm{Pt}}$ for any $n\in F^{-1}(1)$ and $R^{\mathrm{SC}}_{k,\Mr{Pt}}\le C_4 \rho^k$ for any $k\ge0.$ We write $C_2=C_3C_4.$ If $F(n)=0,$ the claim for $n$ obviously follows from the claim for $n-1$ with Lemma \ref{rpt}.2. Otherwise,
				$R_{n,\mathrm{Pt}} \le C_3R^{\mathrm{SC}}_{n-l,\Mr{Pt}}R_{l,\Mr{Pt}} \le  
				2\rho^{m_1(l)}3^{l-m_1(l)}C_2^{m_2(l)}\cdot C_2\rho^{n-l} 
				=2\rho^{m_1(n)}3^{n-m_1(n)}C_2^{m_2(n)}$ by the claim for $l(n),$
			which concludes the proof of the upper bound. The lower bound also follows in a similar way.
			\item We first prove that there exists $m\ge1$ such that if $n\ge0$ satisfies $F(n+1)=\cdots=F(n+m)$ then $R_{n,\Mr{Pt}}\le R_{n+m,\Mr{Pt}}/2.$ Indeed, if $F(n+1)=0$ then $R_{n+m,\Mr{Pt}}\ge R_{n+1,\Mr{Pt}}=3R_{n,\Mr{Pt}}$ by Lemma \ref{rpt}.2. 
			Otherwise, let $l=l(n+m)=l(n)$ and $L$ be the same as in Lemma \ref{cur}, then
			\begin{align*}
				R_{n,\Mr{Pt}}&\le \frac{L}{2}R^{\mathrm{SC}}_{n-l,\Mr{Pt}}R_{l,\Mr{Pt}}\\
				&=\frac{L}{2}\frac{R_{l,\Mr{Pt}}}{R_{l,\Mr{TB}}}\frac{R_{n+m,\Mr{TB}}}{R_{n+m,\Mr{Pt}}}
				\frac{R^{\mathrm{SC}}_{n+m-l,\Mr{TB}}R_{l,\Mr{TB}}}{R_{n+m,\Mr{TB}}}
				\frac{R^{\mathrm{SC}}_{n-l,\Mr{Pt}}}{\rho^{n-l}}\frac{\rho^{n+m-l}}{R^{\mathrm{SC}}_{n+m-l,\Mr{TB}}}
				\rho^{-m}R_{n+m,\Mr{Pt}}
			\end{align*}
			by Lemma \ref{cur}. Therefore Theorem \ref{res}, Proposition \ref{resSC} and the potential theoretic argument shows that $R_{n,\Mr{Pt}}\le R_{n+m,\Mr{Pt}}/2$ if $m$ is sufficiently large. Fix such an $m\ge1,$ then by the definition of $F,$ for any $n\ge0$ there exists $k_n$ such that $n+m^3\le k_n<n+m^3+m$ and  $f(k_n+1)=\cdots=f(k_n+m).$ Therefore
			$R_{n,\mathrm{Pt}}\le R_{k_n,\Mr{Pt}}\le R_{k_n+m,\Mr{Pt}}/2\le R_{n+m^3+2m,\Mr{Pt}}/2$
			for any $n\ge0.$ 
		\end{enumerate}
	\end{proof}
Next, we evaluate resistance metrics between points and check some properties associated with the metric.
	\begin{prop}\label{nxy}
		For $x_1,x_2\in V,$ we denote by $n(x_1,x_2)$ the minimal integer such that there exist $a_j,b_j\in\Mb{Z}\ (j=1,2)$ with the property that 
		\begin{equation}\label{condition}
			x_j\in S_{n(x_1,x_2),a_j,b_j}, \ \Mr{int}(S_{n(x_1,x_2),a_j,b_j})\cap V\ne\emptyset,\ 
			S_{n(x_1,x_2),a_1,b_1}\cap S_{n(x_1,x_2),a_2,b_2}\ne\emptyset
		\end{equation}
		for $j=1,2,$ where $S_{n,a,b}$ is defined by \eqref{inab}. Then $R_{(V,E)}(x_1,x_2) \asymp R_{n(x_1,x_2),\Mr{TB}}$ for any $x,y\in V.$
	\end{prop}
	\begin{proof}
		We first prove $R_{(V,E)}(x_1,x_2) \lesssim R_{n(x_1,x_2),TB}.$ By \eqref{condition}, for $j=1,2$ we can inductively choose $x_{j,k}$ for $0\le k\le n=n(x_1,x_2)$ such that $x_{j,k}\in V\cap  3^k\Mb{Z}^2$ and
		$$\begin{cases}
				|x_j-x_{j,0}|\le2^{-1/2} & \text{if }k=0, \\
				\begin{aligned}
					&\{x_{j,k-1},x_{j,k}\}\subset 3^{k-1}(V_{k-1}+\frac{1}{2}+a+(\frac{1}{2}+b)\iu)\subset V  \\
					&\text{ for some }a,b\in\Mb{Z} 
				\end{aligned}
				&\text{otherwise.}
			\end{cases}$$
	
		It is easy to check that 
		\[R_{(V,E)}(x_{j,k},x_{j,k-1})\le(R_{k-1,\Mr{Pt}}\vee R_{(V_{k-1},E_{k-1})}(p_1,p_3))\le4R_{k-1,\Mr{Pt}},\]
		therefore we obtain $R_{(V,E)}(x_1,x_2)\le 2+8\sum_{k=0}^{n}R_{k,\Mr{Pt}} \le 16MR_{n,\mathrm{Pt}}$
		by Theorem \ref{res}.3, so $R_{(V,E)}(x_1,x_2) \lesssim R_{n(x_1,x_2),\Mr{TB}}$ follows.\par
		We next show  $R_{(V,E)}(x_1,x_2) \gtrsim R_{n(x_1,x_2),\Mr{TB}}.$ By the definition of $n=n(x_1,x_2),$ there exist $a, b\in \Mb{Z}$ and an affine map $\psi$ such that $x_1\in S_{n-1,a,b},\ \psi(S_{n-1,a,b})=\varphi_0(S)$ and $\psi(x_2)\not\in S$
		\begin{figure}[tb]
			\centering
			\begin{tikzpicture}
				\begin{scope}
					\clip (-0.3,-0.3) -- (3.3,-0.3) -- (3.3,3.3) -- (-0.3,3.3) --cycle;
					\foreach \t in {-14,-13,...,21}{
						\path[draw] (0.15*\t, -0.3) -- (0.15*\t + 1.2, 3.3);
					}
					\filldraw[white](0,0) -- (3,0) -- (3,3) -- (0,3) --cycle;
					\begin{scope}
						\clip (1,1) -- (2,1) -- (2,2) -- (1,2) --cycle ;
						\foreach \t in {10,...,18}{
							\path[draw] (0.15*\t, -0.3) -- (0.15*\t - 1.2, 3.3);
						}
					\end{scope}
				\end{scope}
				
				\draw[step=1](-0.2,-0.2) grid (3.2,3.2);
				\coordinate (x0) at (1.3,1.7);
				\filldraw[white] ($(x0)-(0,0.05)$) circle [x radius=0.16, y radius=0.32];
				\filldraw (x0) circle [x radius=0.05, y radius=0.05];
				\draw (x0) node[below]{$x_1$};
				\filldraw[white] (1.8,1.2) circle [x radius=0.1, y radius=0.18];
				\draw (1.8,1.2) node{1};
				\coordinate (y) at (3.2,0.6);
				\filldraw[white] ($(y)-(0,0.05)$) circle [x radius=0.16, y radius=0.32];
				\filldraw (y) circle [x radius=0.05, y radius=0.05];
				\draw (y) node[below]{$x_2$};
				\filldraw[white] (2.6,-0.25) circle [x radius=0.12, y radius=0.2];
				\draw (2.6,-0.25) node{0};
			\end{tikzpicture}
			\caption{$\psi^{-1}(\varphi_0(S))$ and $\Mb{C\setminus}\psi^{-1}(S)$ with values of $f_2(z)$}
			\label{fignxy}
		\end{figure}
		(see  Figure \ref{fignxy}).
		Let $f_1$ be a function on $\cup_{j=0}^8\varphi_j(V_{n-1})$ defined by
	\[f_1(z) =\begin{cases}
				g(\varphi_j^{-1}(z)) & \text{if }z\in\varphi_j(V_{n-1})\text{ for }j=1,3,5,7,\\
				1 &\text{otherwise}
			\end{cases} \] where $g$ be the optimal function for $R_{n-1,\Mr{TB}}.$ We also define $f_2:V\to\Mb{R}$ by
		\[
			f_2(z)=\begin{cases}
				\min\{f_1(\iu^k\psi(z))\mid k=0,1,2,3\} & \text{if }\psi(z)\in S,\\
				0 &\text{otherwise.}
			\end{cases}
	\]
		 Then $f_2(x_1)=1,f_2(x_2)=0$ and $(1/2)\sum_{(x,y)\in E}(f_2(x)-f_2(y))^2\le 12(R_{n-1,\Mr{TB}})^{-1}.$ This concludes the proof.
	\end{proof}
	\begin{defi}[Volume doubling measure]
		 A Borel measure $\nu$ on a metric space $(X,\delta)$ is called a \emph{volume doubling} measure with respect to $\delta,$ if there exists $C>0$ such that $0<\nu(B_\delta(x,2r))\le C\nu(B_\delta(x,r))<\infty$ for any $x\in X$ and $r>0.$

	\end{defi}

	\begin{lem}\label{mu}
		Let $\Mc{V}_n=8^{m_1(n)}5^{n-m_1(n)},$ where $m_1, m_2$ be the same as in Theorem \ref{res}.2. Then
		\begin{enumerate}
			\item $\mu_{(V,E)}(\{y\mid n_{(x,y)\le m}\})\asymp \Mc{V}_m$ for any $x\in V$ and $m\ge0.$
			\item $\mu_{(V,E)}(B_{R_{(V,E)}}(x,R_{(V,E)}(x,y))\asymp \Mc{V}_{n(x,y)}$ for any $x,y\in V.$
			\item $\mu_{(V,E)}$ satisfies the volume doubling condition with respect to $R_{(V,E)}.$
		\end{enumerate}
	\end{lem}
	The proof of Lemma \ref{mu} is straightforward and we omit it. Now, we are ready to prove our aim of this section. 	
	\begin{thm}\label{theds}
		$d_s(V,E)=2\log 5/(\log 3+\log 5).$
	\end{thm}
	\begin{proof}
		By Lemma \ref{mu} and \cite[Theorem 4.27]{Sas21}, there exist a distance $\delta$ on $V$ and $\gamma>1$ such that $\delta\qs V,$ $\delta(x,y)^\gamma\asymp \mu_{(V,E)}(B(x,R_{(V,E)}(x,y)))R_{(V,E)}(x,y) $ for any $x,y\in V,$ and $p(2n,x,x)\asymp 1/\mu_{(V,E)} (B_\delta(x,n^{1/\gamma}))$ for any $x\in V$ and $n\ge0.$
		It is easy to see that there exists $C>0$ such that if $\delta(x,y)^\gamma\le C R_{N,\Mr{Pt}}\Mc{V}_N$ then $n(x,y)\le N.$ Moreover, there exists $C'>0$ such that for any $n\ge0$ there exists $N=N(n)$ satisfying $C'R_{N,\Mr{Pt}}<n\le CR_{N,\Mr{Pt}},$ where $C$ is same as above. Therefore
		\begin{align*}
			&\limsup_{k\to\infty}2\frac{-\log p(2k,x,x)}{\log k} =\limsup_{k\to\infty}2\frac{\log \mu_{(V,E)}(B_\delta(x,k^{1/\gamma}))}{\log k} \\
			\le& \limsup_{N\to\infty} 2\frac{\log \mu_{(V,E)} (\{y\mid \delta(x,y)\le (CR_{N,pt}\Mc{V}_N)^{1/\gamma}\})}{\log (R_{N,\Mr{Pt}}\Mc{V}_N)+\log C'} \\
			\le& \limsup_{N\to\infty} 2\frac{\log \mu_{(V,E)} (\{y\mid n(x,y)\le N\})}{\log (R_{N,\Mr{Pt}}\Mc{V}_N)+\log C'} \\
			\le& \limsup_{N\to\infty} 2\frac{\frac{m_1(N)}{N}\log8+(1-\frac{m_1(N)}{N})\log5}{\frac{m_1(N)}{N}(\log8+\log\rho)+(1-\frac{m_1(N)}{N})(\log5+\log3)+\frac{m_2(N)}{N}\log C_1}\\
			=& 2\frac{\log 5}{\log3+\log5}
		\end{align*}
		because $\lim_{N\to\infty}(m_1(N)/N)=\lim_{N\to\infty}(m_2(N)/N)=0.$ Similarly to the above, we obtain $\liminf_{k\to\infty}-2\log p_{2k}(x,x)/\log k\ge 2\log 5/(\log3+\log5),$ which proves the theorem.
	\end{proof}

\section{Partition satisfying basic framework}\label{secbf}%
In this section, we first introduce the notion of a BF-partition, which was introduced by \cite{Kig2020} for the case of compact metric spaces and extended to the case of infinite graphs in \cite{Sas21}. Then, we also introduce results in \cite{Kig2020, Sas21}, which are necessary to evaluate the Ahlfors regular conformal dimension. See these papers for details.\par
Since the definition of a BF-partition is too complicated to directly understand, we begin by describing the idea of this notion in the case of $(V,E).$ For the convenience of the reader who jumps here from the introduction, we recall that $(V,E)$ is defined by \eqref{defve}. Let $T_*=\{S_{n,a,b}\mid n\ge0, \Mr{int}(S_{n,a,b})\cap V\ne\emptyset\},$ then the essential part of a partition $K$ is defined as $K: T_*\to 2^{V}$ by $K(S_{n,a,b})=V\cap S_{n,a,b}.$ We remark that for any $S_{n,a,b}\in T_*,$ there exists a unique  $S_{n+1,a_*,b_*}\in T_*$ with $K(S_{n,a,b})\subset K(S_{n+1,a_*,b_*}).$ We will consider a graph on $\{S_{n,a,b}\in T_*\}$ for each $n$ such that there exists an edge between $S_{n,a_1,b_1}$ and $S_{n,a_2,b_2}$ if $S_{n,a_1,b_1}\ne S_{n,a_2,b_2}$ and $K(S_{n,a_1,b_1})\cap K(S_{n,a_2,b_2})\ne\emptyset.$\par
Here we return to introduce notions for the definition of a BF-partition. 

\begin{defi}[Tree with a reference point]
	Let $T$ be a countable set and $\pi:T\to T$ be a map such that the following conditions hold.
	\begin{align*}
		 \bullet\ & \pi^n(w)\ne w \text{ for any }n\ge 1\text{ and } w\in T. \\
		 \bullet\ & \text{For any }w,v\in T,\text{ there exist }n,m\ge0\text{ such that }\pi^n(w)=\pi^m(v). 
	\end{align*}
	Fix any $\phi\in T,$ and we call the triplet $(T,\pi,\phi)$ a \emph{bi-infinite tree with a reference point}. 
\end{defi}
As the name shows, $(T,\pi,\phi)$ has a corresponding tree structure on $T$ defined as follows.
\begin{lem}[{\cite[Lemma 3.2]{Sas23}}]\hspace{1pt}\\ \vspace{-1\baselineskip} \begin{enumerate}
		\item Let $b(w,v)=\min\{n\ge0| \pi^n(w)=\pi^m(v)\text{ for some }m\ge0\}$ for $w,v\in T,$ then $\pi^{b(w,v)}(w)=\pi^{b(v,w)}(v).$
		\item Let $\Mc{A}=\{(w,v)\mid \pi(w)=v \text{ or }\pi(v)=w\}$ then $(T,\Mc{A})$ is a tree.
	\end{enumerate}
\end{lem}
For the rest of this section, we assume $(T,\pi,\phi)$ to be a bi-infinite tree with a reference point. We define $[w]=b(w,\phi)-b(\phi,w), \ T_n=\{w\in T\mid [w]=n\}$ and $T^w=\cup_{k\ge0}\pi^{-k}(w)$ for any $w\in T$ and $n\in\Mb{Z}.$
\begin{defi}[Partition]
	Let $(V',E')$ be an infinite graph. A map $K:T\to\{A\subseteq V'\mid \# (A)<\infty\}$ is called a \emph{partition} of $(V',E')$ parametrized by $(T,\pi,\phi)$ if the following conditions hold.
	\begin{align}
		&\Cup_{v\in\pi^{-1}(w)}K(v)=K(w)\text{ for any }w\in T. \tag{$*$}\\
		&\text{If }(w_k)_{k\in\Mb{Z}}\subset T\text{ satisfies }w_k\in T_k\text { and }\pi(w_{k+1})=w_k \text{ for any }k,\text{ then there exist }n_0\tag{$*$}\\
		&\text{and }(x,y)\in E'\text{ such that }K(w_n)=\{x,y\}\text{ for any }n\ge n_0. \notag \\
		&\text{For any }(x,y)\in E',\text{ there exists } w\in T\text{ with }K(w)=\{x,y\}.\label{B01}
	\end{align}
\end{defi}
	Hereafter, we write $K_w$ instead of $K(w)$ for simplicity. Let 
	\begin{align*}
		\Lambda_e&\coloneqq \{w\in T\mid \#(K_{w})=2\text{ and } \#(K_{\pi(w)})>2\}\\
		T_e&\coloneqq\{w\in T\mid T^w\cap\Lambda_e \ne\emptyset \}=\{w\in T\mid \#(K_{\pi(w)})>2\}.
	\end{align*}
\begin{defi}[Basic framework]\label{bf}
	Suppose $\sup_{w\in T_e\setminus \Lambda_e}\#(\pi^{-1}(w))<\infty$, $K$ be a partition of a graph $(V',E')$ parametrized by $(T,\pi,\phi)$ and $\delta$ be a metric on $V'.$ Let
	\[E'_n=\{(w,v)\in (T_n\cap T_e)\times(T_n\cap T_e)\mid K_w\cap K_v\ne\emptyset\text{ and } w\ne v\}\]
	and let $d'_n$ denote the graph distance of $(T_n,E'_n)$ allowing $d'_n(w,v)=\infty.$ We say $K$ satisfies the \emph{basic framework} with respect to $\delta$ if the following conditions hold.
	\begin{align}
		& K_w\ne K_v\text{ for any }w,v\in \Lambda_e \text{ with }w\ne v. \label{B1}\\
		& \text{There exists }\zeta\in(0,1)\text{ such that } \diam_\delta(K_w)\asymp \zeta^{[w]}\text{ for any }w\in T_e.\label{B2}\\
		& \text{There exists }\xi>0\text{ such that  for each }w\in T_e,\text{ there exists }x_w\in K_w\text{ such that if}\label{B3}\\ 
		&(x,y)\in E'\text{ and }\{x,y\}\in B_\delta(x_w,\xi\zeta^{[w]}) \text{ then }\{x,y\}=K_v\text{ for some }v\in T^w.\notag\\
		&\text{Let }\Delta_{m}(x,y)\!=\!\sup\{n\!\mid\! x\in K_w, y\in K_v \text{ and }d'_n(w,v)\le m\text{ for some }w,v\in T_n\}\label{B4}\\
		&\text{then there exists }M_*\in\Mb{N}\text{ with }\delta(x,y)\asymp \zeta^{\Delta_{M_*}(x,y)}\text{ for any }x,y\in V'. \notag\\
		&\textstyle L_*:=\sup_{w\in T_e}\#(\{v\mid (w,v)\in E'_{[w]} \})<\infty \label{B5}
	\end{align}
\end{defi}
\begin{rem}
	To simplify notation, the formulation of the basic framework differs from the one in \cite{Sas21}. We can check the equivalence between these formulations, but skip here.
\end{rem}
To the end of this section, we assume $K$ to be a BF-partition with respect to $\delta$. To use results in \cite{Sas21}, we have to modify $T$ and $K.$ Roughly speaking, we consider the modified version of $K$ as a partition (of a $\sigma$-compact metric space) of the corresponding cable system to $(V',E').$ Suppose $K$ to be a partition of $(V',E')$ satisfying the basic framework parametrized by $(T,\pi,\phi)$. Let 
\[\tau_\zeta\coloneqq\{(x,y,k,m)\mid (x,y)\in E',\ k\ge1\text{ and }1\le m\le 2^{n(k)}\}/{\sim}\]
where $n(k)\ge 0$ with $2^{-n(k)}\le \zeta^k<2^{1-n(k)},$ and $(x_1,y_1,k_1,m_1)\sim(x_2,y_2,k_2,m_2)$ if $x_1=y_2,\ x_2=y_1,\ k_1=k_2$ and $m_1+m_2=2^{k_1}+1.$ Then we define $T_\zeta=T_e\scup\tau_\zeta,$ and $\pi':T_\zeta\to T_\zeta$ by
\[
	\pi'(w)=\begin{cases}
			\pi(w) & \text{if }w\in T_e,\\
			v & \text{if }w=(x,y,1,m)\in \tau_\zeta\text{ and }K_v=\{x,y\},\\
			(x,y,k-1,j) & \begin{aligned} &\text{if }w=(x,y,k,m)\in \tau_\zeta\text{ with }k\ge1 \text{ and}\\
			&j-1<m/(2^{n(k)-n(k-1)})\le j.
			\end{aligned}
		\end{cases}
\]
(Note that $\pi'$ is well-defined by \eqref{B01} and \eqref{B1}.) Then it is easy to show that we can choose $\phi'\in T_\zeta$ with $[\ \cdot\ ]'|_{T_e}\equiv[\ \cdot\ ]|_{T_e}$ where $[\ \cdot\ ]'$ is the height defined with $(T_\zeta,\pi',\phi').$ Thus we consider $[\ \cdot\ ]'$ as an extension of $[\ \cdot\ ]$ and simply write $[\ \cdot\ ]$ instead of $[\ \cdot\ ]'.$ We formally define $K^*: T_\zeta\to\{A\subset V'\}$ by
	\[K^*_w=\begin{cases}
			K_w &\text{if }w\in T_e,\\
			\{x\} &\text{if }w=(x,y,k,1)(=(y,x,k,2^{n(k)}))\in \tau_\zeta\text{ for some }x,y\in V'\text{ and }k\ge1,\\
			\emptyset &\text{otherwise.}
		\end{cases}  \]
Moreover, we also define $(T_\zeta)_n=\{w\in T_\zeta\mid[w]=n\}$ and $E^*_n\subset (T_\zeta)_n\times(T_\zeta)_n$ by
\[E^*_n=\biggl\{(v,w)\biggm\vert \begin{minipage}{280pt}
	$K^*_w\cap K^*_v\ne\emptyset,$ or $v=(x,y,k,m_1),\ w=(x,y,k,m_2)$ for some $x,y\in V',\ k,m_1,m_2\ge1$ with $|m_1-m_2|=1$
\end{minipage} \biggr\}.\]
 \begin{defi}[$p$-spectral dimensions]\label{Ce}
	Let
	\begin{align*}
		N_*&=\limsup_{k\to\infty}\Bigl(\sup_{w\in T_\zeta}\#(\pi^{-k}(w))\Bigr)^{1/k},\\
		\Mc{E}^p_n(f)&=\frac{1}{2}\sum_{(v,w)\in E^*_n}|f(v)-f(w)|^p,\\
		\Mc{C}^M_{w,k}&=\{v\in T_{[w]+k}\mid  d^*_{[w]}(w,\pi^k(v))>N\}\text { and}\\
		\Mc{E}^M_{p,k,w}&=\inf\{\Mc{E}^p_{[w]+k}(f)\mid f:(T_\zeta)_{[w]+k}\to\Mb{R},\ f|_{\pi^{-k}(w)}\equiv1,\ f|_{\Mc{C}^M_{w,k}}\equiv 0\}
	\end{align*}
	for any $n,$ $f:(T_\zeta)_n\to\Mb{R}$, $w\in T$ and $p>0,$ where $d^*_n$ is the graph distance of $((T_\zeta)_n,E^*_n)$ allowing $d^*_n(w,v)=\infty.$ We define the \emph{upper $p$-spectral dimensions} (parametrized by $M$) of the partition $K$ for $p>0$ by
	\begin{equation}\label{uds}
		\ol{d}^s_p(K,M)=p\biggl(1-\frac{\limsup_{k\to\infty}\frac{1}{k}\bigl(\sup_{w\in T_\zeta}\log\Mc{E}^M_{p,k,w}\bigr) }{\log N_*}\biggr)^{-1}
	\end{equation}
	and the \emph{lower $p$-spectral dimensions} $\ul{d}^s_p(K,M)$ for $p>0$ by \eqref{uds} but replacing $\limsup$ by $\liminf.$
\end{defi}
Now we are ready to state Theorem \ref{Smain-} precisely.
\begin{thm}[\cite{Sas21}, Theorem 4.14]\label{Smain}
	If $M\ge M_*$ is sufficiently large (where $M_*$ is the integer in \eqref{B4}), then the following holds.
	\begin{enumerate}
			\item $\inf\{p\mid\liminf_{k\to\infty}(\sup_{w\in T_\zeta}\Mc{E}^M_{p,k,w})=0 \} =\inf\{p\mid\limsup_{k\to\infty}(\sup_{w\in T_\zeta}\Mc{E}^M_{p,k,w})=0\} =\ard(V',\delta).$
			\item If $p>\ard(V',\delta)$ then $p>\ol{d}^s_p(K,M)\ge\ul{d}^s_p(K,M)\ge \ard(V',\delta).$
			\item If $p\le\ard(V',\delta)$ then $p\le\ul{d}^s_p(K,M)\le\ol{d}^s_p(K,M)\le\ard(V',\delta).$
		\end{enumerate}
\end{thm}
\begin{rem}
	This theorem is the discrete version of \cite[Theorem 4.7.9]{Kig2020}. The assumption in Theorem \ref{Smain}.2 and 3 are slightly different from that in \cite{Sas21}, Theorem 4.14] (and also \cite[Theorem 4.7.9]{Kig2020}), but these are justified by \cite[Theorem 4.7.6]{Kig2020}.
\end{rem}

\section{Evaluation of Ahlfors regular conformal dimension}\label{secARC}
In this section, we prove $\ard(V,d)=\ard(\mathrm{SC},|\cdot|_{\Mb{C}})$ and Theorem \ref{theex} using a BF-partition. Recall that $T_*=\{S_{n,a,b}\mid n\ge0, \Mr{int}S_{n,a,b}\cap V\ne\emptyset\}$ and let \[T=T_*\Cup\
\bigl\{(n,S_{0,a,b}/2) \bigm\vert 
			n\ge1,\ a,b\in\Mb{Z}\text{ with }S_{0,a,b}/2\subset S_{0,a_*,b_*}\text{ for some } S_{0,a_*,b_*}\in T_*
	\bigr\}.
\] We define $\pi: T\to T$ by
\[\pi(w)=\begin{cases}
	S_{n+1,a_*,b_*}&\text{if }w=S_{n,a,b}\text{ and }S_{n,a,b}\subset S_{n+1,a_*,b_*},\\
	S_{0,a_*,b_*}&\text{if }w=(1,S_{0,a,b}/2)\text{ and }S_{0,a,b}/2\subset S_{0,a_*,b_*},\\
	(n-1,S_{0,a,b}/2)&\text{if }w=(n,S_{0,a,b}/2)\text{ for }n\ge2.
\end{cases}
\] 
  Then $(T,\pi,S_{0,0,0})$ is a bi-infinite tree with a reference point. Let $K(S_{n,a,b})\coloneqq S_{n,a,b}\cap V'$ for $S_{n,a,b}\in T_*$ and otherwise $K(n,S_{0,a,b}/2)\coloneqq (S_{0,a,b}/2)\cap V',$ then we can see that $K$ is a partition of $V'$ satisfying the basic framework with respect to $|\cdot|_{\Mb{C}}$ as implied in the beginning of Section \ref{secbf}. We also note that $T_{-n}=\{S_{n,a,b}\in T\}$ for $n\ge0.$\par
  \begin{prop}\label{theard}
  	$\ard(V,|\cdot|_{\Mb{C}})=\ard(\mathrm{SC},|\cdot|_\Mb{C}).$
  \end{prop}
\begin{proof}
Let
\begin{align*}
X_k&=\{z\in S\mid \mim(z)=\pm \mre(z),\ \{\mre(z),\mim(z)\}\subset 2^{-1}3^{-k}\Mb{Z}\},\\
V^a_k&=\begin{cases}
	\Phi_{F(a)}\circ\cdots\circ\Phi_{F(a-(k-1))}(V_0)& \text{if } k\le a,\\
	\Phi_{F(a)}\circ\cdots\circ\Phi_{F(1)}(X_{k-a})& \text{if } k>a,
\end{cases}\\
E^a_k&=\{(x,y)\in V^a_k\times V^a_k\mid |x-y|=3^{-n}/\sqrt{2}\}
\end{align*}
for $a,k\ge0.$ For $p>0$ and any graph $(V',E')$ with $V'\subset \Mb{C}$,  we also let
\[\Mc{E}_{p,\Mr{TB}}(V',E')=\inf\biggl\{\frac{1}{2}\sum_{(x,y)\in E'} |f(x)-f(y)|^p \biggm\vert f:V'\to\Mb{R}, f|_{V'\cap S_T}\equiv 1, f|_{V'\cap S_B}\equiv 0 \biggr\}.\] Then Theorem \ref{Smain} and the same symmetry of $(V^a_k,E^a_k)$ as the unit square show that $$\ard(V,|\cdot|_{\Mb{C}})= \inf\{p\vert \limsup_{k\to\infty}\sup_{a\ge0}\Mc{E}_{p,\Mr{TB}}(V^k_a,E^k_a)=0\}.$$
Since $\ard(\mathrm{SC},|\cdot|_{\Mb{C}})=\inf\{p\mid \limsup_{k\to\infty}\Mc{E}_{p,\Mr{TB}}(V^{\Mr{SC}}_k,E^{\Mr{SC}}_k)=0\}$ by \cite[Example 4.6.7 and Theorem 4.6.10]{Kig2020} and symmetry of the Sierpi\'nski carpet, it is easy to see that $\ard(V,|\cdot|_{\Mb{C}})\ge\ard(\mathrm{SC},|\cdot|_{\Mb{C}})\ge1.$ 
To prove $\ard(V,|\cdot|_{\Mb{C}})\le\ard(\mathrm{SC}),$ we fix any $p>\ard(\mathrm{SC},|\cdot|_{\Mb{C}}).$ Then there exists $C>0$ such that if $a,b,k\ge0$ with $a\wedge k>b$ satisfy $F(a)=\cdots=F(a-(b-1))=1$ then $\Mc{E}_{p,\Mr{TB}}(V^k_a,E^k_a)\le C \Mc{E}_{p,\Mr{TB}}(V^{k-b}_{a-b},E^{k-b}_{a-b})\Mc{E}_{p,\Mr{TB}}(V^{\Mr{SC}}_b,E^\Mr{SC}_b).$ Indeed, this claim follows from in the same way as the potential theoretic argument in \cite[Theorem 5.8]{McG} (which is based on\cite[Theorem 4.3]{BB90}) for resistances, because only symmetry of the square, convexity of $f(x)=x^2$ and the fact that the optimal functions are constant on $S_T$ and $S_B$ are used in the argument. Similarly, we may assume that  if $a,b,k\ge0$ with $a\wedge k>b$ satisfy $F(a)=\cdots=F(a-(b-1))=0$ then $\Mc{E}_{p,\Mr{TB}}(V^k_a,E^k_a)\le C \Mc{E}_{p,\Mr{TB}}(V^{k-b}_{a-b},E^{k-b}_{a-b})3^{(1-p)b}$ without loss of generality.
By \cite[Proposition 4.7.5]{Kig2020}, there exists $\gamma<1$ such that $\Mc{E}_{p,\Mr{TB}}(V^{\Mr{SC}}_k,E^\Mr{SC}_k)\lesssim\gamma^k$ for any $k\ge0,$ so there exist $C'>0$ such that
\[\Mc{E}_{p,\Mr{TB}}(V^a_k,E^a_k)\le C'\gamma^{m_1^a(k)}3^{(1-p)(k-m_1^a(k))}C^{m_2^a(k)}\le C'(3^{1-p}\vee \gamma)^k(\sup_{a\ge0} C^{m_2^a(k)})\]
for any $a,k\ge0$, where
\begin{align*}
	m^a_1(k)&\coloneqq\#\{b\mid (a-k)\vee0<b\le a, F(b)=1\},\\
	m^a_2(k)&\coloneqq\#\{b\mid (a-k)\vee0<b\le a, F(b)\ne F(b-1)\}.
\end{align*}
By definition of $F,$ $\lim_{k\to\infty}\sup_a m_2^a(k)/k=0$ so $\lim_{k\to\infty}\sup_a\Mc{E}_{p,\Mr{TB}}(V^a_k,E^a_k)=0.$ Since $p>\ard(\mathrm{SC},|\cdot|_\Mb{C})$ is arbitrary, this means $\ard(G^*,|\cdot|_{\Mb{C}})\le\ard(\mathrm{SC},d_2),$ which concludes the proof.
\end{proof}
\begin{lem}\label{dn}
There exist $C>0$ and $c\in\Mb{N}$ such that for any $x,y,z\in V,$
\begin{enumerate}
\item $d(x,y)\asymp d_{n(x,y)}(p_1,p_5)$ where $d_n$ is the graph distance of $(V_n,E_n).$
\item If $n(x,y)=n(x,z)+1$ then $d(x,y)\le Cd(x,z).$
\item If $n(x,y)\ge n(x,z)+c$ then $d(x,y)\ge 2d(x,z).$
\end{enumerate}
\end{lem}
\begin{proof}
Let $a_n=d_n(p_1,S_B\cap V_n),$ $b_n=d_n(S_T\cap V_n, S_B\cap V_n),$ $c_n=d_n(p_1,p_5),$ and $e_n=d_n(p_1,p_7)$ for any $n\ge0.$ It is obvious that $c_n\wedge e_n\ge a_n\ge b_n.$ Moreover, considering the reflection on $\{z\mid \mre(z)+\mim(z)=0\}$ and $\{z\mid \mim(z)=0\},$ we can obtain $2a_n\ge c_n\vee e_n.$ Additionally, let $l=l(n)$ be same as in Lemma \ref{cur}, then by the reflection of edges, we can see that
\[b_n\ge 3^{n-l+1}a_{l-1}\ge\frac{1}{2}((3^{n-l+1}-2)e_{l-1}+2c_{l-1})\ge\frac{1}{2}a_n,\]
(see  Figure \ref{figbnal}), which shows $a_n\asymp b_n\asymp c_n \asymp e_n$ for any $n\ge0.$\par
Since $b_{n+1}\ge 3b_n,$ we also obtain $d(x,y)\lesssim c_{n(x,y)}$ for any $x,y\in V$ in the same way as Proposition \ref{nxy}. On the other hand, by definition and reflection on edges similar to the above, $d(x,y)\ge b_{n(x,y)-1}$ holds. This and $c_{n+1}\le 5c_n$ shows 1. 2 and 3 immediately follow from 1 and $b_{n+1}\ge 3b_n,$ $c_{n+1}\le 5c_n.$
\end{proof}
\begin{rem}
	$\Mr{diam}(V_n,d_n)\not\asymp 3^n$ for $n\ge0.$ Indeed, $b_{n+1}\ge b_n$ for any $n$ and if $F(n+2)=F(n+1)=0$ and $F(n)=1$ then $b_{n+2}\ge30b_{n-1}$ because $b_{n+2}\ge3a_{n+1}$ and $a_{n+1}\ge10b_{n+2}$ (see Figure \ref{fig10an}).
\end{rem}
\begin{figure}[tb]
	\centering
	\begin{minipage}[b]{0.45\linewidth}
	\centering
		\begin{tikzpicture}[scale=0.45]
		\foreach \b /\B in {0/0, 3/0, 6/0, 9/0, 0/3, 6/3, 9/3}{
			\foreach \a / \A in {0/0, 2/0, 2/2, 0/2, 1/1}{
				\draw (\a+\b,\A+\B) -- (\a+\b+1,\A+\B) --(\a+\b+1,\A+\B+1) --(\a+\b,\A+\B+1) -- cycle;
		}}
		\draw[dashed] (12.2,3) -- (13,3);
		\draw (0.5,0.3) arc (180:270:1); 
		\coordinate[label=right :$3^n(G_{l-1}+\frac{1}{2}+\frac{i}{2})$] (0) at (1.5,-0.9);
		\draw(0,3.5)--(0.2,3.3)--(0.6,3.8)--(0.8,3.7)--(1.4,4.6)--(1.6,4.8)--(2.2,3.6)--(2.4,3.8)--(2.8,2.2)--(3.2,2.6)--(3.6,2.4)--(4.4,1.6)--(4.8,1.8)--(5.2,0.2)--(5.8,0.8)--(6.4,0.4)--(8.4,2.4)--(9.4,2.2)--(9.6,2.4)--(10.2,4.8)--(10.4,4.2)--(10.6,4.4)--(11.4,3.6)--(11.6,2.4)--(12,2.6);
		\draw[dotted,thick](2.6,3)--(2.8,3.8)--(3,3.6);
		\draw[dotted,thick](9.75,3)--(10,2);
		\draw[dotted,thick](11.5,3)--(11.6,3.6)--(12,3.4);
		\foreach \c /\C in {1/4, 2/4, 4/2, 5/1, 7/1, 8/2, 10/4, 11/4}{
			\filldraw (\c, \C) circle [x radius=0.12, y radius=0.12];}
		\draw (10, 2) circle [x radius=0.12, y radius=0.12];
	\end{tikzpicture}\vskip-5pt
	\caption{Decomposition of edge-to-edge path to point-to-edge paths}
	\label{figbnal}
	\end{minipage}
\quad
\begin{minipage}[b]{0.45\linewidth}
	\centering
	 \begin{tikzpicture}[scale=0.5]
		\foreach \b /\B in {0/0, 3/3, 6/0}{
			\foreach \a / \A in {0/0, 1/0, 2/0, 2/1, 2/2, 1/2, 0/2, 0/1}{
				\draw (\a+\b,\A+\B) -- (\a+\b+1,\A+\B) --(\a+\b+1,\A+\B+1) --(\a+\b,\A+\B+1) -- cycle;
			}
			\filldraw (1+\b,1+\B)--(2+\b,1+\B)--(2+\b,2+\B)--(1+\b,2+\B)--cycle;}
		\draw (0.5,0.5) node {1};
		\draw (1.5,0.5) node {2};
		\draw (2.5,1.5) node {3};
		\draw (2.5,2.5) node {4};
		\draw (3.5,3.5) node {5};
		\draw (4.5,3.5) node {6};
		\draw (5.5,3.5) node {7};
		\draw (6.5,2.5) node {8};
		\draw (7.5,2.5) node {9};
		\draw (8.5,2.5) node {10};
		\filldraw (0, 0) circle [x radius=0.2, y radius=0.2];
	\end{tikzpicture}
\caption{Any point-to-edge path must pass 10 blocks.}
\label{fig10an}
\end{minipage}
\end{figure}
\begin{proof}[Proof of Theorem \ref{theex}]
By Lemma \ref{dn} and $|x-y|_\Mb{C}\asymp 3^{n(x,y)},$ it is easy to see that $d\qs|\cdot|_\Mb{C}.$ Therefore $\ard(V,d)=\ard(V,|\cdot|_{\Mb{C}})$ by definition. Since it is known that $\ard(\mathrm{SC}, |\cdot|_{\Mb{C}})\ge 1+(\log 3/ \log 2)$ (see \cite{Ty2000} and \cite{TW2006} for example) and $1+(\log 3/ \log 2)>1.5>2\log 5/(\log 3+\log 5)$, we obtain the desired inequality by Theorem \ref{theds} and Proposition \ref{theard}.
\end{proof}
\begin{rem}
By Theorem \ref{res} and Proposition \ref{nxy}, we can also see $R_{(V,E)}\qs|\cdot|_\Mb{C}$ and so $\ard (V,R_{(V,E)})=\ard(V,|\cdot|_\Mb{C})=\ard(V,d).$
\end{rem}

\section{Cases symmetric for scaling}\label{secSym}

In this section we first prove the following lemma:
\begin{lem}\label{dcube}
	Under the same assumption as Theorem \ref{theth}, there exists a BF-partition $K$ with respect to $d'.$
\end{lem}
For this partition $K$, we later prove $\ol{d}^s_2(K,M)=d_s(V',E',\mu)<2$ for sufficiently large $M$ and Theorem \ref{theth}.\par
Before proving Lemma \ref{dcube}, we introduce a property of a metric space.
\begin{defi}
	A metric space $(X,\delta)$ is called metric doubling if there exists $N>0$ such that for any $x\in X$ and $r>0,$ there exist $\{x_j\}_{j=1}^N$ with $B_\delta(x,2r)\subset \cup_{j=1}^N B_\delta(x_j, r).$
\end{defi}
\begin{proof}[Proof of Lemma \ref{dcube}]
	For $(x,y),(z,w)\in E'$, $H((x,y),(z,w))$ denotes the Hausdorff distance between them, that is,
	\[(d'(x,z)\wedge d'(x,w))\vee(d'(y,z)\wedge d'(y,w))\vee(d'(x,z)\wedge d'(y,z))\vee(d'(x,w)\wedge d'(y,w)).\] 
	Then $H$ is the distance on $E'/{\sim}$ where $(x,y)\sim(z,w)$ if $\{x,y\}=\{z,w\}.$ Since $d'$ is metric doubling (which follows form Ahlfors regularity) and $(V',E')$ is bounded degree, $H$ is also metric doubling. Fix some $w_*\in E'/{\sim}$ and applying \cite[Theorem 2.2]{HK2012}, we obtain $C_1,C_2>0,$ $\zeta\in(0,1),$ $T_k=\{w^k_n\}_{n\in\Mb{N}}\subset E'/{\sim}\ (k\in\Mb{Z})$ and $Q:\scup_{k\in\Mb{Z}}(T)_k\to\{A\subset E'/{\sim}\}$ with the following properties: for any $k,l\in\Mb{Z}$ with $k\le l$ and any $n,m\in\Mb{N}$, $w^k_1=w_*,$ $E'/{\sim}=\sCup_{n\in\Mb{N}}Q_{w^k_n},$  $B_H(w^k_n,C_1\zeta^k)\subset Q_{w^k_n} \subset B_H(w^k_n,C_2\zeta^k),$ and either $Q_{w^l_m} \subset Q_{w^k_n}$ or $Q_{w^l_m}\cap Q_{w^k_n}=\emptyset$ holds. \par
	 Let $T=\scup_{k\in\Mb{Z}}(T)_k.$ For any $w^k_n\in T,$ let $\pi(w^k_n)$ be the unique vertex in $(T)_{k-1}$ and $K_{w^k_n}\coloneqq\{x\mid (x,y)\in Q_{w^k_n} \text{ for some }y\in V'\},$ then $(T,\pi)$ is a bi-infinite tree and $K$ is a partition of $(V',E').$ Moreover, by definition of $H,$ there exist $C_3,C_4$ and $x^k_n\in V'$ such that
	\begin{equation}\label{qball}
		B_{d'}(x^k_n,C_3\zeta^k)\subset K_{w^k_n}\subset B_{d'}(x^k_n,C_4\zeta^k)
	\end{equation}  
	for any $k\le0$ and $n\in\Mb{N}.$ Here we check that $K$ satisfies the basic framework. \par
	Fix any $x,y\in V'$ with $d'(x,y)<\zeta^k,$ then there exist $\{x_j\}_{j=1}^N\subset B_{d'}(x,\zeta^k)$ for some $N$ such that $x_1=x,x_N=y$ and $(x_j,x_{j+1})\in E$ for any $j<N.$ Let
	\begin{equation}M_*:=\sup_{x\in G}\sup_k \#\{w^k_n\in (T)_k\mid B(x,r^k)\cap K_{w^k_n}\ne\emptyset\}-1,\label{mstar}\end{equation} then $M_*<\infty$ because $(V',d')$ is metric doubling. This shows $\Delta_{M_*}(x,y)\ge k,$ and the other properties also follow from metric doubling condition and \eqref{qball}.
\end{proof}
Note that we can choose $\zeta>0$ such that $T_e=\scup_{k\le0}(T)_{k}.$
\begin{prop}\label{lower}
	Let $M_*$ be defined by \eqref{mstar}. For any fixed $M\ge M_*$, $\Mc{E}_{2,k,w}^M\gtrsim \zeta^{k\alpha}$ for any $k$ and $w=w^l_n\in \scup_{l\le -k}T_l.$
\end{prop}
\begin{proof}
	In this proof, we use the idea of modulus of curves. Let $(V_*,E_*,\mu_*)$ be a weighted graph, $A,B\subset V_*$ with $A_*\cap B_*\ne\emptyset.$ We also let 
	\begin{align*}
		\Mc{P}(V_*,E_*,\mu_*,A,B)&=\biggl\{\{x_j\}_{j=0}^m\biggm\vert
		\begin{minipage}{250pt}
			$m\in\Mb{N}, x_0\in A, x_m\in B $ and $(x_j,x_{j+1})\in E$ for any $0\le j< m$
		\end{minipage}\biggr\},\\
		\Mc{F}(V_*,E_*,\mu_*,A,B)&=\biggl\{f:G\to\Mb{R}\biggm\vert \sum_{j=1}^mf(x_j)\ge1 \text{ for any }\{x_j\}_{j=0}^m\in	\Mc{P}(V_*,E_*,\mu_*,A,B)\biggr\},\\
		\Mc{M}(V_*,E_*,\mu_*,A,B)&=\min \{\textstyle\sum_{x\in V_*} f(x)^2\mu(x)\mid f\in \Mc{F}(V_*,E_*,\mu_*,A,B)\} \text{, which exists.}
	\end{align*}
	For the rest of proof, we write $\Mc{Q}_{w,k}^M$ instead of the quintuplet
	\[T_{[w]+k},E_{[w]+k},\mu_{((T)_{[w]+k},E'_{[w]+k})},\pi^{-k}(w),\Mc{C}^M_{w,k}\]
	where $\Mc{C}^M_{w,k}$ is defined in Definition \ref{Ce}, for simplicity. Since $\sup_{x\in V'}\#\{(x,y)\in E'\}<\infty$ by \eqref{p0}, we know that $R_\mu(A,B)^{-1}\asymp\Mc{M}(V',E',\mu,A,B)$ for any $A,B\subset V'.$ Similarly, $\Mc{E}_{2,k,w}^M\asymp \Mc{M}(\Mc{Q}^M_{w,k})$ for any $k\ge0$ and $w\in \scup_{l\le -k}T_l$ by \eqref{B5}.\par
	Fix any $k\ge0,$ $w\in\scup_{l\le -k}T_l,$ and $f\in\Mc{F}(\Mc{Q}^M_{w,k}).$ We next construct a function $g\in\Mc{F}(V',E',\mu,K_w, A^M_w)$ where $A_w^M:=V'\setminus\cup\{K_v\mid v\in(T_\zeta)_{[w]},\ d_{[w]}(w,v)>M\}$, and prove
	\[C^{-1}\zeta^{-[w]\alpha}\le \sum_{x\in V'}g(x)^2\mu(x)\le C\zeta^{-([w]+k)\alpha}\sum_{x\in T_{[w]+k}}f(x)^2\] for some $C>0$ that is independent of $k,w$ and $f,$ which suffices to show this proposition. Let $g_v$ be the optimal function for $\Mc{M}(V',E',\mu,K_v, A^M_v).$ We define $g:V'\to\Mb{R}$ by 
	\[g(x)=\max\bigl\{\tilde f(v)g_v(x)\bigm\vert v\in T_{[w]+k} \text{ such that }x\in \cup \{K_u\mid d_{[w]+k}(u,v)\le M \}\bigr\}\]
	where $\tilde f(v)\coloneqq 2M\max\{f(u)\mid u\in T_{[w]+k},\ d_{[w]+k}(u,v)\le 2M\}.$
	\begin{cla}
		$g\in\Mc{F}(V',E',\mu,K_w, A^M_w).$
	\end{cla}
	\begin{proof}[Proof of the claim]
		Fix any $\{x_j\}_{j=0}^m\in\Mc{P}(V',E',\mu,K_w, A^M_w).$ Let $j_1=\min\bigl\{j\bigm\vert x_j\in \Cap\{A^M_v\mid \pi^k(v)=w\}\bigr\},$ then $j_1$ is well-defined because $A^M_w\subset \Cap\{A^M_v\mid \pi^k(v)=w\}.$ We inductively define $v_a\in T_{[w]+k}$ as the vertex satisfying $\{x_{j_a-1},x_{j_a}\}\subset K_{v_a}$ and $j_a=\min\{j>j_{a-1} \mid x_j\in A^M_{v_{a-1}}\}.$ Then we have some $a^*\ge1$ such that $x_m\in V'\setminus A^M_{v_{a^*}}$ and
		\begin{align*}
			\sum_{j=1}^m g(x_j)\ge \sum_{a=1}^{a^*}\sum_{j=j_{a-1}+1}^{j_a}\tilde f(v_{a-1})g_{v_{a-1}}(x_j)\ge \sum_{a=1}^{a^*}\tilde f(v_{a-1}).
		\end{align*}
		 where $j_0=0$ and $v_0\in \pi^{-k}(w)$ with $x_0\in K_{v_0}.$ Moreover, there exists $\{ v^*_l\}_{l=0}^{(a^*+1)M}\in \Mc{P}(\Mc{Q}^M_{w,k})$ with $v^*_{aM}=v_a$ for $0\le a\le a^*,$ so
		 \begin{align*}
		 	\sum_{a=1}^{a^*}\tilde f(v_{a-1})\ge\sum_{a=1}^{a^*}\sum_{l=(a-1)M+1}^{aM}f(v^*_l) +\sum_{l=a^*M+1}^{(a^*+1)M} f(v^*_l)\ge1.
		 \end{align*}
	\end{proof}
	By \cite[Lemma 6.6 of the arXiv version]{Sas21} and metric doubling property of $d'$, it follows that for fixed $c>1,$ $R(B_{d'}(x,\phi),B_{d'}(x,c\phi))\asymp \phi^\alpha$ for any $\phi\ge1$ and $x\in V'.$ This and \eqref{B2},\eqref{B3} and \eqref{B4} show that $ R(K_w, A^M_w) \asymp \zeta^{[w]\alpha}$ for any $w\in T_e.$ Therefore
	\begin{align*}
		C^{-1}\zeta^{-[w]\alpha}&\le\Mc{M}(V',E',\mu, K_w, A^M_w)\\
		&\le\sum_{x\in V'}g(x)^2\mu(x)\le(L^*+1)^{M+1}\sum_{v\in T_{[w]+k}}\tilde f(v)^2\sum_{x\in V'}g_v(x^2)\mu(x)\\
		&\le 4M^2(L^*+1)^{3M+2}C\zeta^{-([w]+k)\alpha}\sum_{v\in T_{[w]+k}}f(v)^2
	\end{align*} 
	for some $C>0,$ where $L^*$ is defined by \eqref{B5}. This concludes the proof.
\end{proof}
\begin{prop}\label{upper}
	 For any fixed $M\ge M_*$, $\Mc{E}_{2,k,w}^M\lesssim \zeta^{k\alpha}$ for any $k$ and $w=w^l_n\in \scup_{l\le -k}T_l.$
\end{prop}
\begin{proof}
	We use the argument of flow. By \cite[Lemma 2.5]{BCK2005}, there exist $C,C'>1$ such that for any $k\le 0$ and $u,v\in T_k$ with $K_u\cap K_v\ne\emptyset,$ there exists a unit flow $f_{u,v}$ from $u=w^k_a$ to $v$ (as points of $V'$) satisfying 
	$f_{u,v}(x,y)=0$ whenever $\{x,y\}\not\subset B_d(x^k_a, C\zeta^{k}),$ and 
	$2^{-1}\sum_{(x,y)\in E'}f_{u,v}(x,y)^2\mu(x,y)\le C'\zeta^{k\alpha}.$
	Additionally, since $(V',d')$ is metric doubling, 
	\[\sup\{\#\{(w^k_a,w^k_b)\in J_k\mid x\in B_{d'}(x^k_a, C\zeta^{k}) \}\mid x\in V', k\le0\}<\infty\]
	same as Lemma \ref{dcube}. Similarly to Lemma \ref{cur} and Proposition \ref{lower}, this shows
	$C^{-1}\zeta^{[w]\alpha}\le 2^{-1}\sum_{(x,y)\in E'}f(x,y)^2\mu(x,y) \le C\zeta^{([w]+k)\alpha}(\Etkw{w}^M)^{-1}$
	for some $C>0$ and flow $f$ from $K_w$ to $A^M_w.$  This concludes the proof.
\end{proof}
\begin{proof}[Proof of Theorem \ref{theth}]
	Let $M$ be sufficiently large. By Proposition  \ref{upper}, it is easy to see that $\Mc{E}_{2,k,w}^M\lesssim \zeta^{((-[w])\vee 0)\alpha+k-(-[w]\vee 0)}$ for $w\in T_r\setminus \scup_{l\le -k}T_l.$ Considering the resistance restricted on a path, we obtain that $\alpha\le 1.$ Therefore $\lim_{k\to\infty}\sup_{w\in T_r}(\Mc{E}_{2,k,w}^M)^{1/k}=\zeta^\alpha$ by Propositions \ref{lower} and \ref{upper}. On the other hand, $\ol{N}_*=\zeta^{-\beta}$ because $n^\beta\asymp \mu(B_d(x,n)),$ \eqref{B5} and \eqref{qball} hold. Therefore  $ \ol{d}^s_2(K,M)=2\beta/(\alpha+\beta)=d_s(V',E',\mu)<2$, where the third equation follows from \cite[Theorem 1.3]{BCK2005}. This with Theorem \ref{Smain} suffices to prove the statement.
\end{proof}

\appendix
\section{Proof of $R_{n,\Mr{TB}}^{\mathrm{SC}}\gtrsim R_{n,\Mr{Pt}}^{\mathrm{SC}}$}\label{app}
Let $D_n=(\Phi_1)^n(\{p_0,p_2,p_4,p_6,p_8\})$ and $B_n=\{(x,y)\in D_n\times D_n \mid |x-y|=3^{-n}2^{-1}\}.$
We say $g:D_n\to\Mb{R}$ is harmonic on $A\subset D_n$ if $g(x)=\sum_{y: (x,y)\in D_n} g(y)/\mu_{(D_n,B_n)}(x)$ for any $x\in A.$
In this section, we will use the following fact which can be proved in a similar way to \cite[Theorem3.1]{BB89} (see also \cite[Theorem 4.4]{BB99}). 
\begin{prop}[Harnack's inequality for the graphical Sierpi\'nski carpet]\label{Har}
There exists $C>0$ such that any $n\ge0$ and any nonnegative harmonic function $f$ on 
$\{z\in D_n\mid \mre(z) \vee \mim(z)\ne\frac{1}{2}\}$ satisfy $cf(x)\ge f(y)$ for any $x,y\in D_n\cap\varphi_1(S).$
\end{prop}
\begin{figure}[tb]
	\centering
	\begin{minipage}[b]{0.45\linewidth}
		\centering
		\begin{tikzpicture}[scale=1.2]
			\begin{scope}
				\clip (3,0) --(0,0)--(0,3) --cycle;
				\foreach \t in {-14,-13,...,21}{
					\path[draw] (0.15*\t, -0.3) -- (0.15*\t + 1.2, 3.3);
				}
			\end{scope}
			\draw[dashed] (0,0) --(3,0) -- (3,3) -- (0,3) --cycle;
			\draw (-0.1,3.1) -- (3.1,-0.1) node[below ]{$\mim(z)=-\mre(z)$};
			\coordinate[label=right :$1$] (a1) at (3,2.833) ;
			\coordinate[label=above :$1$] (a2) at (2.833,3) ;
			\foreach \n in {1,2,3,4}
			\coordinate (b\n) at ($(3,3)-0.167*(\n,1)$);
			\foreach \n in {5,6,7}
			\coordinate (b\n) at ($(3,3)-0.167*(1,\n-3)$);
			\draw (a1) -- (b4);
			\draw[dotted,thick] (b4) -- ($(b4)-(0.3,0)$);
			\draw[dotted,thick] (b7) -- ($(b7)-(0,0.3)$);
			\draw (a2) -- (b7);
			\draw ($(b3)-(a1)+(3,3)$) -- ($(b3)+(a1)-(3,3)$);
			\draw ($(b6)-(a2)+(3,3)$) -- ($(b6)+(a2)-(3,3)$);
			\foreach \n in {1,2}
			\filldraw (a\n) circle [x radius=0.05, y radius=0.05];
			\foreach \n in {1,...,7}
			\filldraw[white] (b\n) circle [x radius=0.01, y radius=0.01];
			\filldraw[white] (1,1) circle [x radius=0.15, y radius=0.24];
			\draw (1,1) node{0};
		\end{tikzpicture}
		\vskip-5pt
		\caption{Areas where $\tilde{R}_{n,\tr}$ denotes the resistance}\label{Rtri}
	\end{minipage}
	\begin{minipage}[b]{0.45\linewidth}
		\centering
		\begin{tikzpicture}[scale=1.2]
			\draw(0,0) --(3,0) -- (3,3) -- (0,3) --cycle;
			\filldraw(1.1,1.1) -- (1.1,1.9) -- (1.9,1.9) -- (1.9,1.1) --cycle ;
			\filldraw(0.25,0.25)--(0.5,0.25) -- (0.5,0.5) -- (0.25,0.5) --cycle;
			\filldraw(2.5,0.25)--(2.75,0.25) -- (2.75,0.5) -- (2.5,0.5) --cycle;
			\filldraw(2.5,2.5)--(2.75,2.5) -- (2.75,2.75) -- (2.5,2.75) --cycle;
			\coordinate (a1) at (0,0);
			\coordinate (a2) at (0.25,0.25);
			\coordinate (a3) at (0.5,0);
			\coordinate (a4) at (0.75,0.25);
			\coordinate (a5) at (2.25,0);
			\coordinate (a6) at (2.5,0.25);
			\coordinate (a7) at (2.75,0);
			\coordinate (a8) at (3,0.25);
			\coordinate (a9) at (2.75,0.5);
			\coordinate (a10) at (3,0.75);
			\coordinate (a11) at (2.75,2.25);
			\coordinate (a12) at (3,2.5);
			\coordinate (a13) at (2.75,2.75);
			\coordinate (a14) at (3,3);
			\draw (a1) -- (a2) -- (a3) --(a4);
			\draw[dotted,thick](0.75,0.125) -- (1,0.125) ;
			\draw[dotted,thick](2.25,0.125) -- (2,0.125) ;
			\draw[dotted,thick](2.875,0.75) -- (2.875,1) ;
			\draw[dotted,thick](2.875,2.25) -- (2.875,2) ;
			\draw (a5)--(a6) -- (a7) -- (a8) -- (a9) --(a10);
			\draw (a11) -- (a12) -- (a13) -- (a14);
			\foreach \n in {1,...,14}
			\filldraw (a\n) circle [x radius=0.05, y radius=0.05];
		\end{tikzpicture}
		\caption{$2\cdot3^k$ points in $G^{\mathrm{SC}}_{n+k}$ whose resistance with the nearest point in themselves is less than $R_{n,\Mr{Pt}}^{\mathrm{SC}}$}
		\label{TI}
	\end{minipage}
\end{figure}
\begin{proof}[Proof of Proposition \ref{resSC}]
	As we mentioned before, it suffices to show $R_{n,\Mr{TB}}^{\mathrm{SC}}\gtrsim R_{n,\Mr{Pt}}^{\mathrm{SC}}.$ Let $\tilde{R}_{n,\Mr{TB}}$ and $\tilde{R}_{n,\tr}$ denote
	\begin{align*}
		\tilde{R}_{n,\Mr{TB}}&=R_{(D_n,B_n)}(D_n\cap S_T, D_n\cap S_B),\\
		\tilde{R}_{n,\tr}&=R_{(D_n,B_n)}\biggl(\biggl\{p_1-\frac{\iu}{2\cdot3^n}, \ p_1-\frac{1}{2\cdot3^n}\biggr\},D_n\cap\{z\mid \mre(z)+\mim(z)=0\}\biggr)
	\end{align*}
	respectively (see  Fig. \ref{Rtri}). It is easy to see that $\tilde{R}_{n,\Mr{TB}}\asymp R_{n,\Mr{TB}}^{\mathrm{SC}}$ and $\tilde{R}_{n,\tr}\asymp R_{n,\Mr{Pt}}^{\mathrm{SC}}:$ the latter follows from $f^{\Mr{SC}}_n(z)-1/2=-f^{\Mr{SC}}_n(-z)+1/2,$ where $f^{\Mr{SC}}_n$ is the optimal function for $R^{\Mr{SC}}_{n,\Mr{Pt}}.$\par Let $g_n$ be the optimal function for $\tilde{R}_{n,\tr},$ then $g_n$ is harmonic on 
	\[D_n\cap\{z\mid \mre(z)+\mim(z)>0\}\setminus\{p_1-\frac{\iu}{2\cdot3^n},\ p_1-\frac{1}{2\cdot3^n}\}.\]
	Here we fix $k\ge0$ such that $\tilde{R}_{n+k,\tr}\ge2\tilde{R}_{n,\tr}$ for any $n\ge0$ ($R_{n+k,\Mr{Pt}}^{\mathrm{SC}}\gtrsim R_{n,\Mr{Pt}}^{\mathrm{SC}} R_{k,\Mr{TB}}^{\mathrm{SC}}$ and $\rho^n\asymp R_{n\Mr{TB}}^{\mathrm{SC}}$ assure the existence of such $k$), and define the function $h_n$ on $D_{n+k+2}$ by
	\[h_n(x)=\begin{cases}
		\tilde{R}_{n+k+2,\tr}\cdot g_{n+k+2}(x)-\tilde{R}_{n+2,\tr}\cdot g_{n+2}(\varphi^{-k}_1(x)) & \text{if } x\in \varphi^k_1(S),\\
		\tilde{R}_{n+k+2,\tr}\cdot g_{n+k+2}(x) & \text{otherwise.}
	\end{cases}\]
	Since $g_n(p_1-2^{-1}3^{-n})-g_n(p_1-2^{-1}3^{-n}(1+\iu))=(\tilde{R}_{n,\tr})^{-1}/2,$ $h_n$ is harmonic on $D_{n+k+2}\cap\varphi^k_1(\{z\mid \mre(z)+\mim(z)>0 \})$, especially on $ D_{n+k+2}\cap\varphi_1^{k+1}(S).$ Moreover, since
	$h_n(x)=\tilde{R}_{n+k+2,\tr}\cdot g_{n+k+2}(x)-0\ge0$ on $D_{n+k+2}\cap\varphi_1^k(\{\mim(z)=-\mre(z)\}),$ $h_n$ is nonnegative.
	Therefore there exists $C>0$ such that
	$Ch_n(x)\ge h_n(p_1-2^{-1}3^{-n})=\tilde{R}_{n+k+2,\tr}-\tilde{R}_{n+2,\tr}\ge\tilde{R}_{n+2,\tr}$ for any $x\in D_{n+k+2}\cap\varphi_1^{k+2}(S)$ because of Proposition \ref{Har}. This shows
	\begin{align*}
		&R_{(D_{n+k+2},B_{n+k+2})}(\varphi^{k+2}_1(D_n),\varphi^{k+2}_5(D_n))\\
		\ge&C^{-2}(\tilde{R}_{n+2,\tr})^2(2((\tilde{R}_{n+k+2,\tr})^{2-1}+(\tilde{R}_{n+2,\tr})^{2-1}))^{-1}\\
		\ge&C^{-2}\frac{1}{4}(\tilde{R}_{n+2,\tr})^2(\tilde{R}_{n+k+2,\tr})^{-1}.
	\end{align*}
	On the other hand, by the potential theoretic argument as in \cite[Theorem4.3]{BB90},
	\begin{align*}
		&R_{(D_{n+k+2},B_{n+k+2})}(\varphi^{k+2}_1(D_n),\varphi^{k+2}_5(D_n))\\
		\lesssim &
		R_{(D_{k+2},B_{k+2})}(\varphi^{k+2}_1(D_0),\varphi^{k+2}_5(D_0))\tilde{R}_n
	\end{align*}
	for any $n\ge0.$ We also obtain that $(2\cdot3^k-1)R_{n,pt}^{\mathrm{SC}}\ge R_{n+k,pt}^{\mathrm{SC}}$ by the triangle inequality of the resistance metric (see  Fig. \ref{TI}). Therefore we obtain
	\[R_{n,\Mr{TB}}^{\mathrm{SC}}\gtrsim R_{n-2,\Mr{TB}}^{\mathrm{SC}}\gtrsim\tilde{R}_{n-2,\Mr{TB}}\gtrsim(\tilde{R}_{n,\tr})^2/\tilde{R}_{n+k,\tr} \gtrsim (R_{n,\Mr{Pt}}^{\Mr{SC}})^2/R_{n+k,\Mr{Pt}}^{\Mr{SC}}\gtrsim R_{n,\Mr{Pt}}^{\mathrm{SC}}\] for any $n\ge2.$
\end{proof}

\bigskip
\address{
	Research Institute for Mathematical Sciences \\ 
	Kyoto University \\
	Kitasirakawa Oiwakecho \\
	Kyoto 606-8502 \\
	Japan
}
{ksasaya@kurims.kyoto-u.ac.jp}
\end{document}